\documentclass{amsart}
\usepackage{pstricks}
\usepackage{graphicx}

\newtheorem{thm}{Theorem}[section]

\newtheorem{defn}[thm]{Definition}
\newtheorem{lemma}[thm]{Lemma}

\newtheorem{cor}[thm]{Corollary}
\newtheorem{remark}[thm]{Remark}
\newtheorem{example}[thm]{Example}

\usepackage{amsmath}
\usepackage{amsxtra}
\usepackage{amscd}
\usepackage{amsthm}
\usepackage{amsfonts}
\usepackage{amssymb}
\usepackage{eucal}

\newcommand{\cA}{\mathcal A}

\newcommand{\cT}{\mathcal T}
\newcommand{\bbT}{\mathbb T}
\newcommand{\bbF}{\mathbb F}
\newcommand{\bbY}{\mathbb Y}
\newcommand{\bbA}{\mathbb A}
\newcommand{\bbB}{\mathbb B}
\newcommand{\bbC}{\mathbb C}
\newcommand{\bbU}{\mathbb U}
\newcommand{\bbV}{\mathbb V}
\newcommand{\bbW}{\mathbb W}
\newcommand{\bbZ}{\mathbb Z}
\newcommand{\cZ}{\mathcal Z}
\newcommand{\cC}{\mathcal C}
\newcommand{\cP}{\mathcal P}

\newcommand{\g}{{\mathfrak{g}}}

\newcommand{\Z}{{\mathbb Z}}

\newcommand{\Q}{{\mathbb Q}}

\newcommand{\bm}{{\mathbf m}}
\newcommand{\bM}{{\mathbf M}}

\newcommand{\bx}{{\mathbf x}}
\newcommand{\by}{{\mathbf y}}
\newcommand{\bw}{{\mathbf w}}
\newcommand{\bz}{{\mathbf z}}
\newcommand{\vf}{{\mathbf f}}
\newcommand{\bg}{{\mathbf g}}

\newcommand{\al}{{\alpha}}

\newcommand{\wB}{{\widetilde{B}}}
\newcommand{\wG}{{\widetilde{G}}}
\newcommand{\wS}{{\widetilde{S}}}
\newcommand{\wx}{{\widetilde{\mathbf x}}}

\newcommand{\wt}{{\overline{T}}}

\numberwithin{equation}{section}
\numberwithin{figure}{section}
\begin{document}
\title{Positivity of the T-system cluster algebra}
\author{Philippe Di Francesco and Rinat Kedem}
\address{PDF: Insitut de Physique Th\'eorique du Commissariat \`a l'Energie Atomique, Unit\'e de recherche associe\'ee du CNRS, CEA Saclay/IPhT/Bat 774, F-91191 Gif sur Yvette Cedex, France. e-mail: philippe.di-francesco@cea.fr}
\address{RK: Department of Mathematics, University of Illinois, Urbana, IL 61821. e-mail: rinat@illinois.edu}
\date{\today}
\begin{abstract}
We give the path model solution for the cluster algebra variables of the $A_r$ $T$-system 
with generic boundary conditions. The solutions are 
partition functions of (strongly) non-intersecting paths on weighted graphs. The 
graphs are the same as those constructed for the $Q$-system
in our earlier work, and depend on the seed or initial data in terms of which the 
solutions are given. 
The weights are ``time-dependent'' where ``time'' is the extra parameter which 
distinguishes the $T$-system from the $Q$-system, 
usually identified as the spectral parameter in the context of representation 
theory. The path model is alternatively described on a graph with 
non-commutative weights, and cluster mutations are interpreted as non-commutative
continued fraction rearrangements. As a consequence, the solution is a positive Laurent polynomial 
of the seed data. 
\end{abstract}

\maketitle

\section{Introduction}
In this paper we study solutions of the $T$-system associated to the Lie 
algebras $A_r$, which we write in the following form:
\begin{equation}\label{Tsys}
T_{\al,j,k+1}T_{\al,j,k-1} = T_{\al,j+1,k}T_{\al,j-1,k}+T_{\al+1,j,k}T_{\al-1,j,k},
\end{equation}
where $ j,k\in \Z,\ \al\in I_r=\{1,...,r\},$ and
with boundary conditions 
\begin{equation}\label{boundary}
T_{0,j,k}=T_{r+1,j,k}=1,\quad j,k\in \Z.
\end{equation} 
We consider these equations to be discrete evolution equations for the 
commutative variables $\{T_{\al,j,k}\}$ in the direction of the discrete variable 
$k$. 

Originally, this relation appeared as the fusion relation for the commuting 
transfer matrices  of the generalized Heisenberg model \cite{BR,KNS} 
associated with a simply-laced Lie algebra $\g$, where it is written in the form
\begin{equation}\label{tsys}
\wt_{\al,j,k+1}\wt_{\al,j,k-1} = \wt_{\al,j+1,k}\wt_{\al,j-1,k}-\prod_{\beta\neq\al} 
\wt_{\beta,j,k}^{-C_{\beta,\al}},
\end{equation}
with appropriate boundary conditions. The matrix $C$ is the symmetric Cartan 
matrix of one of the Lie algebra of type $ADE$. Our relation \eqref{Tsys} is 
obtained by a rescaling  of the variables $\wt_{\al,j,k}$ and specializing to the 
Cartan matrix  of  $A_r$.

With special initial condition at $k=0$, it has been proved that the solutions to 
\eqref{tsys} are the $q$-characters \cite{FR} of the Kirillov-Reshetikhin 
modules of the affine Lie algebra $U_q(\widehat{sl}_{r+1})$ \cite{Nakajima}. 

The $T$-system also appears in several other contexts. Of particular 
relevance here is the fact  \cite{LWZ} that the system is a discrete integrable 
equation, the discrete Hirota equation. It is therefore to be expected that the 
system has a complete set of integrals of motion, and that it is exactly 
solvable.
This equation also appears in a related combinatorial context, as the 
octahedron equation, which was studied by \cite{KnutsonTao,Speyer}. 

In this paper, we do not impose any special boundary conditions, but express 
the general solution of the $T$-system in terms of arbitrary initial conditions. 
For example, initial conditions can be chosen by specifying the values of the 
parameters $T_{\al,j,k}$ at $k=0$ and $k=1$, or a more exotic boundary can 
be specified. To solve the system, we  use a path model which is a simple 
generalization of the path model we constructed for the solutions of the $Q$-
system of $A_r$\cite{DFK3,DFK4}.  

In our previous work, we constructed a set of path models, and proved that the 
solutions of the $Q$-system of $A_r$ \cite{KR},
$$
Q_{\al,k+1}Q_{\al,k-1}=Q_{\al,k}^2+Q_{\al-1,k}Q_{\al+1,k},
\quad Q_{0,k}=Q_{r+1,k}=1; k\in\Z, \al\in I_r,
$$
are the generating functions for paths on a positively weighted graph, where 
the weights are a function of the initial conditions. 

With special initial conditions at $k=0$ and $k=1$ (together with a rescaling as 
in \eqref{tsys} which restores the minus sign in the second term on the right 
hand side of the $Q$-system), the solutions are the characters the finite-
dimensional, irreducible modules of $A_r$ with highest weights which are 
multiples of one of the fundamental weights. 

Note that this $Q$-system is obtained by ``forgetting" the spectral parameter 
$j$ in Equation \eqref{Tsys}. Thus the $T$-system can be regarded as an 
affinization or $q$-deformation of the $Q$-system, and the path model we 
present here is therefore a deformation of the path model for the $Q$-system.

Without fixing any special initial conditions, it was shown in \cite{Ke} that the 
solutions of the $Q$-system are cluster variables in a cluster algebra \cite{FZ}.  
We showed in \cite{DFK2} that all $Q$-systems, corresponding to any simple 
Lie algebra, can be formulated as cluster algebras. Thus, the solution of the 
$Q$-system in terms of the statistical model allowed us to prove the positivity 
conjecture of \cite{FZ} for these cluster variables. In fact, as we showed in 
\cite{DFK4}, the solutions are related to the totally positive matrices of 
\cite{FZtotalposit} corresponding to pairs of coxeter elements.

Similarly, we showed in \cite{DFK2} that a large class of equations which we 
call generalized bipartite $T$-systems can be 
formulated as cluster algebras. Equation \eqref{Tsys} is perhaps the simplest 
example of such a system. Motivated by our statistical model introduced in 
\cite{DFK3}, we introduce a path model which provides us with the solution to 
the $T$-system, in terms of a set of initial conditions, as the partition function 
of a path model with time-dependent (or non-commutative) weights. Here, we 
refer to the variable normally identified as the spectral parameter as the time 
parameter, as it is a natural interpretation from the point of view of paths.

This paper is organized as follows. In Section \ref{sec:cluster}, we review the 
necessary definition of a cluster algebra. We recall our formulation \cite{DFK2} 
of $T$-systems as cluster algebras. We describe the conserved quantities of 
the $T$-system in terms of discrete Wronskian determinants in Section 
\ref{sec:wronskians}. We define a generalized notion of hard particle models 
on a graph in Section \ref{sec:conserved} and identify the conserved quantities as hard 
particle partition functions on a specific graph. In Section \ref{firpositsec}, we use our 
conserved quantities to write the solutions of the $T$-system as the partition functions 
of paths on a weighted graph. 
The weight of a step in a path depends on the order in which the steps are taken, that is, 
the weights are time-dependent. The solutions are written as functions of the 
fundamental initial data, and the graph is the same as the one used in the 
$Q$-system solution. Positivity of the $T$-system solutions in terms of the 
fundamental seed variables follows from this formulation. 

To prove the positivity in terms of other seeds, we give a formulation of our model 
in terms of non-commutative weights in Section \ref{sec:otherpaths}. We are then 
able to describe the  solutions of the $T$-system as a function of other seed data 
as partition functions on new graphs with weights which depend on the mutated 
seeds. The key to the construction is an operator version of the fraction rearrangement 
lemmas used in \cite{DFK2}. These rearrangements are equivalent to mutations in the 
case of the $Q$-system. Here, they are equivalent to compound mutations. 
We are thus able to write the $T$-system solution explicitly in terms of its initial
data, for a subset of cluster seeds. 

This paper should be considered as a (special case of) non-commutative generalization 
of our work on the solutions of $Q$-system \cite{DFK3,DFK4}. In particular, 
the graphs on which we build our path models are the same as for the $Q$-system, 
and the only difference is the time-dependence or non-commutativity of the weights. 
The various key properties, such as the rearrangement lemmas for continued 
fractions and the generalization of  the Lindstr\"om-Gessel-Viennot theorem 
for strongly non-intersecting paths, all have straightforward non-commutative 
counterparts which are used here.

\noindent{\bf Acknowledgements:} 
P.D.F.'s research is supported in part by the 
ANR Grant GranMa, the
ENIGMA research training network MRTN-CT-2004-5652,
and the ESF program MISGAM.  R.K.'s 
research is supported by NSF grant DMS-0802511. R.K. thanks IPhT at CEA/
Saclay for their kind hospitality. We also acknowledge the hospitality of the
Mathematisches Forschungsinstituts Oberwolfach (RIP program), 
where this paper was completed.

\section{$T$-systems as cluster algebras}\label{sec:cluster}

\subsection{Cluster algebras}
We use the following definition of a cluster algebra \cite{FZ,Zel04}, slightly 
specialized to suit our needs in this paper. 

Let $S\subset \wS$ be two discrete sets (possibly infinite) and consider the 
field $\mathcal F$ of rational 
functions over $\Q$ in a set of independent variables indexed by $\wS$.  

We define a {\em seed} in 
$\mathcal F$ to be a pair $(\widetilde{\bx}, \widetilde{B})$, where $\wx=\{x_m: 
m\in \widetilde{S}\}$ is a set of 
commuting variables, and $\widetilde{B}$ is an integer matrix, with rows 
indexed by $\widetilde{S}$ and 
columns indexed by $S$. The matrix $B$, which is the square submatrix of $
\wB$ made up of the rows of $\widetilde{B}$ indexed by $S$, is skew 
symmetric. 

The {\em cluster} of the seed $(\widetilde{x},\widetilde{B})$ is the set of 
variables $\{x_m :  m\in S\}$, and the {\em coefficients} are the set of variables 
$\{x_m:  m\in \widetilde{S}\setminus S\}$.

Next, we define a {\em seed mutation}. For any $m\in S$, a mutation in the 
direction $m$, $\mu_m: (\widetilde{x},\widetilde{B})\mapsto (\widetilde{x}',
\widetilde{B'})$, is a discrete evolution of the seed. Explicitly,
\begin{itemize}
\item The mutation $\mu_m$ leaves $x_n$ with $n\neq m$ invariant, and 
updates the variable $x_m$ only, via the exchange relation
\begin{equation}\label{exchange}
x_m' = x_m^{-1} \left(\prod_{n\in \widetilde{S}} x_n^{[\widetilde{B}_{n,m}]_+}+ 
\prod_{n\in \widetilde{S}} x_n^{[-\widetilde{B}_{n,m}]_+}\right)
\end{equation}
where $[n]_+= {\rm max}(n,0)$.
\item The exchange matrix $\widetilde{B}'$ has entries  
\begin{equation}
\wB'_{i,j} = \left\{ \begin{array}{ll} -\wB_{i,j} & \hbox{if $i=m$ or $j=m$} ;\\
\wB_{i,j} + {\rm sign}(\wB_{i,m})[\wB_{i,m}\wB_{m,j}]_+) & \hbox{otherwise}.
\end{array}\right.
\end{equation}
\end{itemize}

Note that we only define mutations for the set $S$, and not for the coefficient 
set $\wS\setminus S$. That is, coefficients do not evolve.

Fix a seed $(\widetilde{\bx},\widetilde{B})$ and consider the orbit $\mathcal X
\subset\mathcal F$ of the cluster variables under all combinations of the 
mutations $\mu_m,\ m\in S$. The cluster algebra is the $\Z[\mathbf c^{\pm1}]$-
subalgebra of $\mathcal F$ generated by $\mathcal X$, where $\mathbf c$ is 
the common coefficient set of the orbit of the seed.

\begin{remark} The particular system which we solve in this paper does not 
require us to have a coefficient set, that is, we can set $S=\widetilde{S}$. 
However, to make more direct contact with representation theory, it is desirable 
to have the coefficient set be enumerated by the roots of the Lie algebra. In 
this context, we need to set the values of the coefficients to the special points 
$ -1$.
\end{remark}

Cluster algebras can be considered to be discrete dynamical systems, which is 
the point of view we adopt in this paper. 

\subsection{bipartite $T$-systems as cluster algebras}\label{sec:Tcluster}
In this section we review some of the definitions of Appendix B of \cite{DFK2}, 
where generalized bipartite $T$-systems were shown to have a cluster algebra 
structure. 

\begin{defn}
A generalized bipartite $T$-system is a recursion relation for the commuting, 
invertible variables $\{T_{\al,j;k}\}$, where $\al\in I_r$ and $j,k\in \Z$, of the 
form
\begin{equation}\label{genT}
T_{\al,j;k+1}T_{\al,j;k-1} = T_{\al,j+1;k}T_{\al,j-1;k} + q_\al \prod_{j'} \prod_{\al'} 
(T_{\al',j';k})^{A_{\al',\al}^{j',j}}
\end{equation}
where $A$ is an incidence matrix, that is, a symmetric matrix with positive 
integer entries.
\end{defn}

The matrix $A$ is generally of infinite size, unless special boundary conditions 
are imposed on the system which truncate the range of the variables $j$. We 
do not impose such boundary conditions in this paper, although they are 
clearly of interest \cite{HL,N}. The symmetry of $A$ is required for  the {\em 
bipartite} property to hold (see below). $T$-systems which are not bipartite can 
also be defined, and in that case, the matrix $A$ is not symmetric.

\begin{example}
The first example of such a $T$ system is the one described in \eqref{tsys}. In 
that case, we take the matrix $A$ to be as follows:
\begin{equation}\label{Amatrix}
A_{\al,\beta}^{j,j'} = \mathcal I_{\al,\beta} \delta_{j,j'} ,
\end{equation}
where $\mathcal I_{\al,\beta}=C-2 I$ is the incidence matrix of the Dynkin 
diagram associated with a simply-laced Lie algebra $\g$. The coefficients $q_
\al$ are all set to be $-1$. However, it is always possible to renormalize the 
variables so that $q_\al=1$ in these cases \cite{Ke}, and we use this approach 
here. 

In particular, if $\g=A_r$, $(\mathcal I)_{\al,\beta} = \delta_{\al,\beta+1}+
\delta_{\al,\beta-1}$. This is the case we solve in this paper.
\end{example}

We note that another example of generalized $T$-systems appeared in the 
context of preprojective algebras and the categorification program of 
\cite{GLS}. The explicit connection was made in \cite{DFK2}, Example 4.4.

Finally, define the (possibly infinite) matrix $P$ with entries 
\begin{equation}\label{Pmat}
P_{\al,\beta}^{j,l}=\delta_{\al,\beta}(\delta_{i,j+1}+\delta_{i,j-1}).
\end{equation}
 Then we can rewrite \eqref{genT} as
\begin{equation}\label{gengenT}
T_{\al,j;k+1}T_{\al,j;k-1} = \prod_{\al,j} T_{\beta,j'}^{P_{\beta,\al}^{j',j}} + q_\al 
\prod_{j'} \prod_{\al'} (T_{\al',j';k})^{A_{\al',\al}^{j',j}}.
\end{equation}
In the systems considered in \cite{DFK2}, we allowed the matrix $P$ to be a 
matrix with positive integer entries, such that it commutes with the matrix $A$, 
together with another condition on the sum of its entries (see Lemma 
\ref{lemmaTcluster} below). Such a system is also a generalized bipartite $T$-
system.

\subsection{Cluster algebra structure}
We recall the formulation found in Appendix B of \cite{DFK2} of the cluster 
algebra associated with generalized (bipartite) $T$-systems. 

In the notations of Section \ref{sec:cluster}, let $S= (I_r \sqcup \overline{I}_r)
\times Z$, and $\wS=S\sqcup I_r'$. Each set $I_r, \overline{I}_r$ and $I_r'$ is 
just the set with $r$ elements. For convenience, if $\al\in I_r$, then by $
\overline{\al}$ we mean the $\al$th element of $\overline{I}_r$, etc.

We define the {\em fundamental} seed $(\wx,\wB)_0$ as follows. The variables 
$\wx_0$ are
\begin{eqnarray}\label{Tclusterseed}
x_{\al,j}&=& T_{\al,j;0},\ (\al\in I_r, j\in \Z); \nonumber\\
x_{\overline{\al},j} & =& T_{\al,j;1},\ (\overline{\al}\in \overline{I}_r, j\in \Z);
\nonumber \\
x_{\al'} &=& q_\al , \ \al'\in I_r';\\
\end{eqnarray}
The elements of the set $\{x_{\al,j}\}\sqcup \{x_{\overline{\al},j}\}$ are the 
cluster variables and $\{x_{\al'}\}$ are the coefficients. The exchange matrix of 
the fundamental seed is defined as follows:
\begin{eqnarray}\label{Texchangeseed}
&&B_{\al,j;\beta,l} = 0 , \ (\al,\beta \in I_r,\ j,l\in \Z), \quad 
B_{\overline{\al},j;\overline{\beta},l} = 0 , \ (\overline{\al},\overline{\beta} \in 
\overline{I}_r,\ j,l\in \Z),\nonumber\\
&&  
B_{\al,j;\overline{\beta},l} = -P_{\al,\beta}^{j,l} + A_{\al,\beta}^{j,l} =
-B_{\overline{\beta},l;\al,j}\nonumber\\
&& \wB_{\al';\beta,j}=-\wB_{\al',\overline{\beta},j}=-\delta_{\al,\beta}.
\end{eqnarray}
The last equation above denotes the entries of the extended $B$-matrix, 
corresponding to the coefficients, which do not mutate. The matrices $A,P$ 
are those of equation \eqref{gengenT} for the generalized $T$-system. 

\begin{example}
In the case of the $A_r$ system \eqref{Tsys}, we have the matrix $A$ as in 
\eqref{Amatrix}, $P$ as in \eqref{Pmat} and $q_\al=1$. In that case we do not 
need to include the coefficients $q_\al$, and the matrix $\widetilde{B}$ is equal 
to the matrix $B$. To recover the original $T$-system \eqref{tsys}, we take $q_
\al=-1$. 
\end{example}

It is clear that each of the mutations $\mu_{\al,j}$ and $\mu_{\overline{\al},j}$ 
exchanges one of the cluster variables in $\wx_0$ via one of the $T$-system 
equation relations \eqref{gengenT}.
The mutation $\mu_{\al,j}$ acts on $\wx_0$ as one of the $T$-system 
evolutions \eqref{gengenT}, where we specialize to $k=1$: $\mu_{\al,j}(T_{\al,j;
0})=T_{\al,j;2}$. Similarly, 
$\mu_{\overline{\al},j} T_{\al,j;1}=T_{\al,j;-1}$ is a $T$-system equation 
specialized to $k=0$.

Quite generally, if $B_{a,b}=0$ then $\mu_a\circ \mu_b = \mu_b \circ \mu_a$. 
Since $B_{\al,j;\beta,l}=0$ for all $\al,\beta\in I_r$ and $j,l\in \Z$,
when acting on the initial seed $(\wx,\wB)_0$, the mutations $\mu_{\al,m}$ 
commute with each other for all $\al,m$. Similarly the mutations $
\mu_{\overline{\al},m}$ also commute among themselves.

Therefore we can define the compound mutations $$\mu:=\prod_{\al,m} 
\mu_{\al,m},\quad \overline{\mu}:=\prod_{\overline{\al},m} 
\mu_{\overline{\al},m}$$
which act on  $(\wx,\wB)_0$. More generally,
Define $(\wx,\wB)_{2k}$ to be the seed with $x_{\al,j}=T_{\al,j;2k}, 
x_{\overline{\al},j}=T_{\al,j;2k+1}$ and 
$\wB_{2k}=\wB$. Define $(\wx,\wB)_{2k+1}$ to be the seed with 
$x_{\al,j}=T_{\al,j;2k+2}, x_{\overline{\al},j}=T_{\al,j;2k}$ and 
$\wB_{2k+1}=-\wB$. 
Then it is clear that $\mu(\wx_{2k})=\wx_{2k+1}$: Each mutation $\mu_{\al,j}$ 
mutates the variable $T_{\al,j;2k}$ into the variable $T_{\al,j;2k+2}$. Similarly, 
it is easy to check that $\overline{\mu}(\wx_{2k})=\wx_{2k-1}$, $\mu(\wx_{2k
+1})=\wx_{2k}$ and $\overline{\mu}(\wx_{2k+1})=\wx_{2k+2}$.

The following statement is  Lemma 4.6 of \cite{DFK2}:
\begin{lemma}\label{lemmaTcluster}
Assume that the matrix $A$ commutes with the matrix $P$, and that 
\begin{equation}\label{Prest}
\sum_k P_{\al,\beta}^{kj} = 2 \delta_{\al,\beta}
\end{equation}
for any $j$. 
Then the cluster algebra $\mathcal X$ which includes the seed $(\wx,\wB)_0$ 
as in \eqref{Tclusterseed}, \eqref{Texchangeseed} includes all the solutions of 
the $T$-system \eqref{gengenT}. All the $T$-system relations are exchange 
relations in this cluster algebra.
\end{lemma}
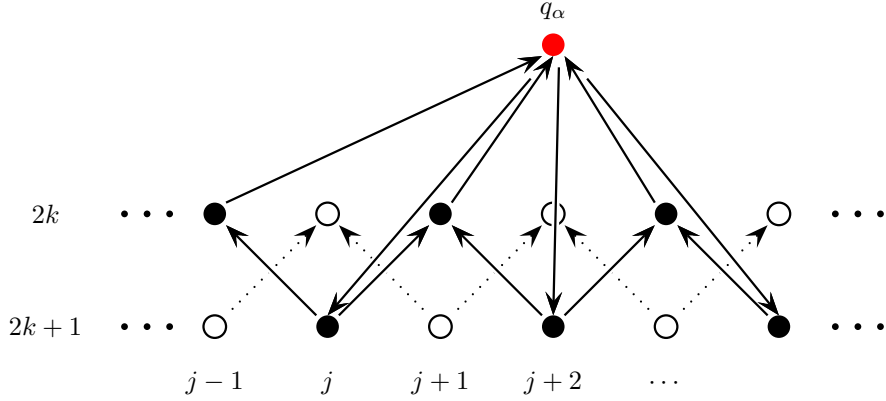
\begin{figure}\label{alphaslicefigure}
\begin{center}
\psset{unit=1.5mm,linewidth=.3mm,dimen=middle}
\begin{pspicture}(-20,-10)(70,30)
\multips(10,0)(20,0){3}{\rput(0,0){\pscircle*(0,0){1}}}
\multips(0,0)(20,0){3}{\rput(0,0){\pscircle(0,0){1}}}
\multips(0,10)(20,0){3}{\rput(0,0){\pscircle*(0,0){1}}}
\multips(10,10)(20,0){3}{\rput(0,0){\pscircle[border=1pt](0,0){1}}}
\rput(10,-5){$j$}
\rput(0,-5){$j-1$}
\rput(20,-5){$j+1$}
\rput(30,-5){$j+2$}
\rput(40,-5){$\cdots$}
\multips(55,0)(2,0){3}{\rput(0,0){\pscircle*(0,0){0.3}}}
\multips(55,10)(2,0){3}{\rput(0,0){\pscircle*(0,0){0.3}}}%
\multips(-8,0)(2,0){3}{\rput(0,0){\pscircle*(0,0){0.3}}}
\multips(-8,10)(2,0){3}{\rput(0,0){\pscircle*(0,0){0.3}}}
\rput(-15,0){$2k+1$}
\rput(-15,10){$2k$}
\multips(0,0)(20,0){3}{\psline[linestyle=dotted,arrowsize=3pt
    4]{->}(1,1)(9,9)}
\multips(0,0)(20,0){2}{\psline[linestyle=dotted,arrowsize=3pt 4]{->}(19,1)(11,9)}
\multips(0,0)(20,0){3}{\psline[border=1pt,arrowsize=3pt 4]{->}(9,1)(1,9)}
\multips(0,0)(20,0){2}{\psline[border=1pt,arrowsize=3pt 4]{->}(11,1)(19,9)}
\rput(30,25){\pscircle*[linecolor=red](0,0){1}}
\psline[border=1pt,arrowsize=3pt 4]{<-}(10,1)(28,22)
\psline[border=1pt,arrowsize=3pt 4]{<-}(30,1)(30.5,23)
\psline[border=1pt,arrowsize=3pt 4]{<-}(50,1)(33,22)
\psline[border=1pt,arrowsize=3pt 4]{->}(1,11)(29,24)
\psline[border=1pt,arrowsize=3pt 4]{->}(21,11)(30,24)
\psline[border=1pt,arrowsize=3pt 4]{->}(39,11)(31,24)
\rput(30,28){$q_\alpha$}
\end{pspicture}
\caption{A slice of the quiver graph of $\wB$, corresponding to constant $\al$. 
The nodes in the strip are labeled by $(j,k)$ of $T_{\al,j;k}$. The two 
subgraphs with even and odd $j+k$ decouple in this slice, so we illustrate the 
only the connectivity of  nodes of the same parity to node $q_\al$. The 
mutation $\mu$ reverses all arrows connected to $q_\al$.}
\end{center}
\end{figure}
To prove this Lemma, we need
\begin{lemma}
$$
 \mu\left( (\wx,\wB)_{2k}\right) = (\wx,\wB)_{2k+1},\quad \overline{\mu}
\left((\wx,\wB)_{2k} \right) = (\wx,\wB)_{2k-1}.
$$
\end{lemma}
\begin{proof}
In light of the preceding discussion, 
all that needs to be proved is that $\mu(\wB)=\overline{\mu}(\wB)=-\wB$. Let $
\wB'=\mu(\wB)$. Then, since $B_{\al,j;\beta,k}=0$, we have
\begin{itemize}
\item $\mu_{\al,i}(B_{\beta,j;\gamma,k})=  {\rm sign} (B_{\beta,j;\alpha,i})
[B_{\beta,j;\alpha,i}B_{\alpha,i;\gamma, k}]_+=0$;
\item $\mu_{\al,i}(B_{\overline{\beta},j;\gamma,k})=-B_{\overline{\beta},j;
\gamma,k}$ if $(\al,i)=(\gamma,k)$, and is otherwise unchanged, since if $
(\al,i)\neq(\gamma,k)$,
$$\mu_{\al,i}(B_{\overline{\beta},j;\gamma,k})=B_{\overline{\beta},j;\gamma,k} 
+ {\rm sign}(B_{\overline{\beta},j;\al,i})[B_{\overline{\beta},j;\al,i} B_{\al,i;
\gamma,k}]_+=B_{\overline{\beta},j;\gamma,k}.$$
Similarly, $\mu_{\al,i}(B_{\beta,j;\overline{\gamma},k})= -B_{\beta,j;
\overline{\gamma},k}.$
\item  Recall the restriction that $[P,A]=0$. Then
$$
\mu(B_{\overline{\beta},j;\overline{\gamma},k})=
\sum_{\al,i} {\rm sign}(B_{\overline{\beta},j;\al,i})[B_{\overline{\beta},j;\al,i} 
B_{\al,i;\overline{\gamma},k}]_+\\
= (P A - A P)_{\beta,\gamma}^{j,k}=0.
$$
\item We have $\mu_{\al,i}(B_{\beta';\gamma,k})=-B_{\beta';\gamma,k}$, and 
otherwise,
if $(\al,i)\neq (\gamma,k)$ then $\mu_{\al,i}$
$$
\mu_{\al,i}(\wB_{\beta';\gamma,k})=\wB_{ \beta';\gamma,k}+ \sum_{\al,i} {\rm 
sign}(\wB_{\beta';\al,i})[\wB_{\beta';\al,i} B_{\al,i;\gamma,k}] = \delta_{\beta,
\gamma},
$$
so that $\mu(\wB_{\beta';\gamma,k})=-\wB_{\beta';\gamma,k}$.
\item Finally, using the restriction \eqref{Prest} on the summation of elements 
of $P$,
$$
\mu(\wB_{\beta';\overline{\gamma},k}) = \delta_{\beta,\gamma} +\sum_{\al,i}
{\rm sign}(\wB_{\beta';\al,i})[\wB_{\beta';\al,i} B_{\al,i;\overline{\gamma},k}]_+ 
= \delta_{\beta,\gamma}-\sum_\al\delta_{\al,\beta}
\sum_i P_{\al,\gamma}^{i,k} = -\delta_{\beta,\gamma}.
$$
\end{itemize}
In the quiver graph corresponding to $\wB$, the last two statements are about 
how nodes $x_{\al,j}=T_{\al,2k}$ and  $x_{\overline{\al},j}=T_{\al,j;2k+1}$ are 
connected to node $x_{\beta'}=q_\beta$. If $\al\neq \beta$, they are not 
connected, and if $\al=\beta$, the connectivity is illustrated in Figure 
\ref{alphaslicefigure} and the mutations in Figure \ref{mutationfigure}.

\begin{figure}\label{mutationfigure}
\begin{center}
\psset{unit=.9mm,linewidth=.3mm,dimen=middle}
\begin{pspicture}(0,-5)(70,80)
\rput(0,50){\rput(10,0){\pscircle*(0,0){1}}
\rput(10,-3){$j$}
\rput(0,10){\pscircle*(0,0){1}}
\rput(0,-3){$j-1$}
\rput(20,-3){$j+1$}
\rput(20,10){\pscircle*(0,0){1}}
\rput(10,25){\pscircle*[linecolor=red](0,0){1}}
\rput(13,25){$q_\alpha$}
\psline[arrowsize=3pt 4]{->}(9.5,1)(0.5,9)
\psline[arrowsize=3pt 4]{->}(10.5,1)(19.5,9)
\psline[arrowsize=3pt 4]{->}(1,11)(9,24)
\psline[arrowsize=3pt 4]{->}(19,11)(11,24)
\psline[arrowsize=3pt 4]{->}(10,24)(10,1)
\psline[arrowsize=3pt 4]{->}(25,10)(45,10)}
\rput(35,63){$\mu_{\alpha,j-1}\mu_{\alpha,j+1}$}
\rput(50,50){\rput(10,0){\pscircle*(0,0){1}}
\rput(0,10){\pscircle*(0,0){1}}
\rput(20,10){\pscircle*(0,0){1}}
\rput(10,25){\pscircle*[linecolor=red](0,0){1}}
\psline[arrowsize=3pt 4]{<-}(9.5,1)(0.5,9)
\psline[arrowsize=3pt 4]{->}(10.5,1)(19.5,9)
\psline[arrowsize=3pt 4]{<-}(1,11)(9,24)
\psline[arrowsize=3pt 4]{->}(19,11)(11,24)}
%\psline[arrowsize=3pt 4]{->}(10,24)(10,1)}
\psline[arrowsize=3pt 4]{->}(57,47)(37,37)
\rput(47,45){$\mu_{\alpha,j}$}
\rput(20,10){\rput(10,0){\pscircle*(0,0){1}}
\rput(0,10){\pscircle*(0,0){1}}
\rput(20,10){\pscircle*(0,0){1}}
\rput(10,25){\pscircle*[linecolor=red](0,0){1}}
\psline[arrowsize=3pt 4]{<-}(9.5,1)(0.5,9)
\psline[arrowsize=3pt 4]{<-}(10.5,1)(19.5,9)
\psline[arrowsize=3pt 4]{<-}(1,11)(9,24)
\psline[arrowsize=3pt 4]{<-}(19,11)(11,24)
\psline[arrowsize=3pt 4]{<-}(10,24)(10,1)}
\end{pspicture}\caption{ The local action of the mutation $\mu$ on a section of 
the quiver graph. The compound mutation reverses all arrows connected to 
$q_\al$.}
\end{center}
\end{figure}
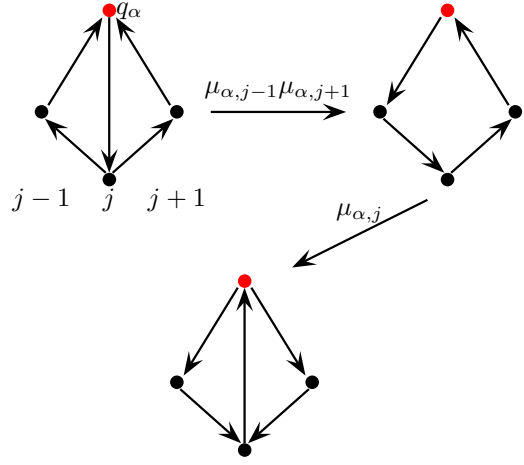

We have shown that $\mu(\wB)=-\wB$. The proof that $\overline{\mu}(\wB)=-
\wB$ is similar.
\end{proof}

Thus, we have shown that all the variables $T_{\al,j;k}$ appear in the cluster 
algebra, in fact, within a bipartite graph composed of the nodes reached from $
(\wx,\wB)_0$ via combinations of the compound mutations $\mu$ and $
\overline{\mu}$ only.

In this paper, we study the $A_r$ $T$-system solutions in terms of the 
fundamental seed cluster $\wx_0$. The result will be an explicit interpretation 
of the solutions as partition functions of paths on a graph whose weights which 
are positive monomials in the variables $\wx_0$. This will imply the positivity 
property \cite{FZ} for the cluster variables $T_{\al,j;k}$: They can be expressed 
as Laurent polynomials with non-negative coefficients in terms of the initial 
data.

\section{Basic properties of the $T$-system}\label{sec:wronskians}

From here on, we specialize the discussion to the $T$-system
\eqref{Tsys}.  Note that the equation \eqref{Tsys} is a three-term
recursion in the index $k$, and allows to determine all the
$\{T_{\al,j,k+1}\}_{\al\in I_r,j\in\Z}$ in terms of the $\{
\{T_{\al,j,k},T_{\al,j,k-1}\}_{\al\in I_r,j\in\Z}$.  We wish to first
study the solution $T_{\al,j,k}$ to Equation \eqref{Tsys} in terms of
the ``fundamental" initial data $\bx_0=(T_{\al,j,0},T_{\al,j,1})_{\al\in
I_r,j\in \Z}$, that is, $\bx_0$. The techniques used in this section
are a straightforward generalization of the methods used for the
$Q$-system in \cite{DFK3}.  We therefore present the proofs of the
theorems in the Appendix, as they use standard techniques in the
theory of determinants.

\subsection{Discrete Wronskians and conserved quantities}

We can express the subset of variables $\{T_{\alpha,j,k}: j,k\in \Z, \al>1\}$  as 
polynomials of the variables in the set $\{ T_{1,j,k}:j,k\in\Z\}$, {\em cf} 
\cite{KNS}:
\begin{thm}\label{elim}
\begin{equation}\label{wronsal}
T_{\al,j,k}=\det_{1\leq a,b\leq \al} \, \left(T_{1,j-a+b,k+a+b-\al-1}\right), \quad \al
\in I_r, \ j,k\in \Z
\end{equation}
\end{thm}
The proof of this theorem uses the standard Pl\"ucker relations, and is similar 
to the case of the $Q$-system. We therefore present the details of the proof in 
the Appendix, Section \ref{wronskianApp}.

If we consider $\al=r+1$ in Equation \eqref{wronsal}, since $T_{r+1,j;k}=1$, we 
have the polynomial relation among the variables $\{T_{1,j;k}\}$:
\begin{equation}\varphi_{j,k}\equiv \left\vert
\begin{matrix}
T_{1,j,k-r} & T_{1,j-1,k+1-r} & \cdots & T_{1,j-r+1,k-1} & T_{1,j-r,k}\\
T_{1,j+1,k+1-r} & T_{1,j,k+2-r} & \cdots & T_{1,j-r+2,k} & T_{1,j-r+1,k+1}\\
\vdots & \vdots & \ddots & \vdots & \vdots \\
T_{1,j+r-1,k-1} & T_{1,j+r-2,k} & \cdots & T_{1,j,k+r-2} & T_{1,j-1,k+r-1}\\
T_{1,j+r,k} & T_{1,j+r-1,k+1} & \cdots & T_{1,j+1,k+r-1} & T_{1,j,k+r}
\end{matrix} \right\vert =1
\end{equation}
This is the ``equation of motion" for the system. Since $\varphi_{j,k}$ is  a 
discrete  Wronskian
determinant, it remains constant  for solutions of a difference equation.
The difference equation can be  found by taking the difference of two 
Wronskians and arguing that
a non-trivial linear combination of its columns must vanish. 

\begin{thm}\label{recut}
We have the following linear recursion relations
\begin{equation}\label{linrecone}
\sum_{b=0}^{r+1} T_{1,j-b,k+b} (-1)^b c_{r+1-b}(j-k)=0 \quad j,k\in \Z
\end{equation}
where the coefficients $c_{r+1-b}(j-k)$ depend only on the difference $j-k$, 
with $c_0(m)=c_{r+1}(m)=1$ for all $m\in\Z$, and:
\begin{equation}\label{linrectwo}
\sum_{a=0}^{r+1} T_{1,j+a,k+a} (-1)^a d_{r+1-a}(j+k)=0 \quad j,k\in \Z
\end{equation}
where the coefficients $d_{r+1-a}(j+k)$ depend only on the sum $j+k$, 
with $d_0(m)=d_{r+1}(m)=1$ for all $m\in\Z$.
\end{thm}
 Such linear recursion relations can be obtained by noting that $T_{r+2,j,k}=0$ 
and expanding the corresponding Wronskian determinant along the first row or 
column. The key fact to be proven is that the minors depend only on the 
difference $j-k$ or the sum $j+k$.
The proof is presented in the Appendix, Section \ref{linearrecursion}.

By analogy with the case of the $Q$-systems \cite{DFK3,DFK4}, we may still 
call the variables $c_b(k)$ and $d_b(k)$ integrals
of motion of the $T$-system, as they depend on one less variable than $T$.
 Moreover, they can be expressed entirely
in terms of the fundamental
initial data for the $T$-system, $\wx_0$.

\begin{example}
In the $A_1$ case, we have
\begin{eqnarray*}
T_{1,j,k} -c_1(j-k) T_{1,j-1,k+1}+T_{1,j-2,k+2} &=&0\\
T_{1,j,k} -d_1(j+k) T_{1,j+1,k+1}+T_{1,j+2,k+2} &=&0
\end{eqnarray*}
with the integrals of motion
\begin{eqnarray*}
c_1(j)&=&{T_{1,j,0}\over T_{1,j-1,1}}+{1\over T_{1,j-1,1}T_{1,j-2,0}}
+{T_{1,j-3,1}\over T_{1,j-2,0}}\\
d_1(j)&=&{T_{1,j,0}\over T_{1,j+1,1}} +{1\over T_{1,j+1,1}T_{1,j+2,0}}
+{T_{1,j+3,1}\over T_{1,j+2,0}}
\end{eqnarray*}
\end{example}

An explicit expression for the conserved quantities  of 
Theorem \ref{recut} is as Wronskian determinants with a ``defect":
\begin{lemma}\label{conwron}
The conserved quantities $c_m(j)$ ($m=0,1,...,r+1$, $j\in \Z$) of Equation 
\eqref{linrecone} are 
\begin{equation}\label{consW}
c_m(j)=\det_{1\leq a\leq r+1\atop 1\leq b\leq r+2,\ b\neq r+2-m}\,  \left(T_{1,j+n
+a-b,n+a+b-2}\right)
\end{equation}
for any $n\in \Z$.
\end{lemma}
Again the proof uses the standard techniques, and is found in Section 
\ref{conservedwronskian} of the Appendix.

\section{Conserved quantities and hard particles}\label{sec:conserved}

\subsection{Recursion relations for conserved quantities}
The conserved quantities \eqref{consW} satisfy linear recursion relations, 
which allow us to express them in terms of 
the initial data $\bx_0$.  We use recursion
relations on the size $r$, so we first relax the boundary conditions $T_{r
+1,j,k}=1$ for all $j,k\in\Z$. 

Consider the $A_{\infty/2}$
$T$-system:
\begin{eqnarray}\label{ainfTsys}
\ \ t_{\al,j,k+1}t_{\al,j,k-1}&=&t_{\al,j+1,k}t_{\al,j-1,k}+t_{\al+1,j,k}t_{\al-1,j,k}, 
\quad  t_{0,j,k}=1,\ (j,k\in \Z, \, \al\in\Z_{>0}).
\end{eqnarray}
Solutions of this system are expressible in terms of the initial data $(t_{\al,j,
0},t_{\al,j,1})_{\al\in \Z_{>0},j\in\Z}$.
By definition,  $T_{\al,j,k}=t_{\al,j,k}$ if we impose the boundary condition
$t_{r+1,j,k}=1$ for all $j,k\in\Z$. 

The proof of Theorem \ref{elim} does not involve
the boundary condition $T_{r+1,j;k}=1$, so the determinant expression for 
$t_{\al,j;k}$ still holds: 
\begin{equation}
t_{\al,j,k}=\det_{1\leq a,b\leq \al} \, \left(t_{1,j+a-b,k+a+b-\al-1}\right), \quad 
\al>1.
\end{equation}

Define the Wronskians of size $N$ with a defect in position $N-m$:
\begin{equation}\label{defcamjk}
c_{N,m,j,k}=\det_{1\leq a\leq N\atop 1\leq b\leq N+1,\, b\neq N+1-m}\, 
\left(t_{1,j+a-b,k+a+b-N-1}\right),
\end{equation}
where $c_{N,m,j,k}=0$ if $m>N$ or $m<0$. 

Note that 
$c_{N,0,j,k}=T_{N,j,k}$, and $c_{N,N,j,k}=T_{N,j-1,k+1}$ by Theorem 
\ref{elim}.
If we impose the second boundary condition of the $T$-system on the $t$'s, 
then 
$c_{r+1,m,j,k}=c_{m}(j-k+r)$. 

\begin{lemma}\label{lemrec}
The Wronskians with a defect $c_{\al,m,j,k}$ satisfy the following recursion 
relations:
\begin{eqnarray}
t_{\al-1,j-1,k-1}\, c_{\al-1,m,j,k}&=&t_{\al,j-1,k}\, c_{\al-2,m-1,j,k-1}+t_{\al-1,j,k}\, 
c_{\al-1,m,j-1,k-1}\label{firstrec}\\
t_{\al-1,j-1,k}\, c_{\al,m,j,k}&=&t_{\al,j-1,k+1}\, c_{\al-1,m-1,j,k-1}+t_{\al,j,k}\, 
c_{\al-1,m,j-1,k}\label{secrec}
\end{eqnarray} 
for $\al\geq 2$ and $m,j,k\geq 1$.
\end{lemma}
\begin{proof} The first equation \eqref{firstrec}
follows from the Desnanot-Jacobi relation \eqref{desna}, 
with $N=\al$, $i_1=1$, $i_2=\al$,
$j_1=\al-m$, $j_2=\al$, for the matrix $M$ with entries $M_{a,b}=T_{1,j+a-
b-1,k+a+b-\al-1}$, 
$a,b=1,2,...,\al$. 

The second equation \eqref{secrec} follows from the Pl\"ucker relation
\eqref{specplu}, with $N=\al$, and the $N\times (N+2)$
matrix $P$ with entries $P_{a,1}=\delta_{a,\al}$, and $P_{a,b}= T_{1,j+a-b,k+a
+b-\al-1}$ for
$b=2,3,...,\al+2$ and $a=1,2,...,\al$, and by further
picking $a_1=1$, $a_2=2$, $b_1=\al+2-m$, and $b_2=\al+2$.
\end{proof}

\begin{thm}\label{recutheo}
The Wronskians with a defect $c_{\al,m,j,k}$ defined in Equation 
\eqref{defcamjk}
are uniquely determined by the following recursion relation, for $\al\geq 2$:
\begin{eqnarray}\label{recuho}
&&t_{\al-1,j-1,k-1}\, t_{\al-1,j,k}\, c_{\al,m,j,k-1}=
t_{\al-1,j-1,k-1}\, t_{\al,j-1,k}\, c_{\al-1,m-1,j+1,k-1} \nonumber \\
&&\qquad \qquad\qquad \qquad+t_{\al,j,k-1}\, t_{\al-1,j,k}\, c_{\al-1,m,j-1,k-1}+ 
t_{\al,j,k-1}\, t_{\al,j-1,k}
\, c_{\al-2,m-1,j,k-1}
\end{eqnarray}
and the boundary conditions $c_{0,m,j,k}=\delta_{m,0}$, for all $j,k\in \Z$ and 
$c_{1,m,j,k}=\delta_{m,0}\, t_{1,j,k}+\delta_{m,1}\, t_{1,j-1,k+1}$
for all $m,j,k\in\Z$.
\end{thm}
\begin{proof}
Using Equation \eqref{firstrec},
 the second line in \eqref{recuho} is equal to 
$t_{\al,j,k-1}t_{\al-1,j-1,k-1}c_{\al-1,m,j,k}$. Canceling
the overall factor $t_{\al-1,j-1,k-1}$, we must prove that
\begin{equation}\label{reduone}
t_{\al-1,j,k}\, c_{\al,m,j,k-1}-( t_{\al,j-1,k}\, c_{\al-1,m-1,j+1,k-1}+t_{\al,j,k-1}\, 
c_{\al-1,m,j,k})=0.
\end{equation}
Multiplying the l.h.s. of  \eqref{reduone} by $t_{\al,j+1,k}$ and using 
\eqref{ainfTsys}
we have
\begin{eqnarray*}\label{calcul}
&&t_{\al,j+1,k}t_{\al-1,j,k}c_{\al,m,j,k-1}
-(t_{\al,j,k+1}t_{\al,j,k-1}-t_{\al+1,j,k}t_{\al-1,j,k}) c_{\al-1,m-1,j+1,k-1}\nonumber 
\\
&&-t_{\al,j+1,k}t_{\al,j,k-1}c_{\al-1,m,j,k}\nonumber \\
&=&t_{\al-1,j,k}(t_{\al,j+1,k} c_{\al,m,j,k-1}+t_{\al+1,j,k} c_{\al-1,m-1,j+1,k-1})
\nonumber \\
&&-t_{\al,j,k-1}(t_{\al,j,k+1}c_{\al-1,m-1,j+1,k-1}+t_{\al,j+1,k}c_{\al-1,m,j,k})
\nonumber \\
&=&t_{\al,j,k-1}(t_{\al-1,j,k}c_{\al,m,j+1,k}-t_{\al,j,k+1}c_{\al-1,m-1,j+1,k-1}-
t_{\al,j+1,k}c_{\al-1,m,j,k})=0
\end{eqnarray*}
where we have simplified the third line by use of \eqref{firstrec}, and finally 
used \eqref{secrec}. 
Equation  \eqref{recuho} follows.

Equation \eqref{recuho} is a three-term linear recursion relation in the variable 
$\al$
and therefore has a unique solution $c$ given the initial conditions at $
\al=0,1$.
Moreover, these initial conditions are identical to those for \eqref{defcamjk},
hence this solution coincides with the definition \eqref{defcamjk} for all $
\al,m,j,k$.
\end{proof}

Let us define:
\begin{equation}\label{quantities}
C_{\al+1,m}(j,k)= {c_{\al+1,m,j+k-\al,k}\over t_{\al+1,j+k-\al,k}}, \qquad \al\geq 
0,\quad m,j,k\in\Z.
\end{equation}
These satisfy  $C_{\al+1,0}(j,k)=1$ and 
$C_{\al+1,\al+1}(j,k)=t_{\al+1,j+k-\al-1,k+1}/t_{\al+1,j+k-\al,k}$. 
The conserved quantities
of the $A_r$ $T$-system are obtained by imposing the boundary condition 
$t_{r+1,j,k}=1$, in which
case: $c_m(j)=C_{r+1,m}(j,k)$ for any $j\in\Z$, and $m=0,1,2,...,r+1$, 
independently
of $k\in\Z$.

\begin{cor}\label{correcu}
The quantities $C_{\al+1,m}(j,k)$ of eq.\eqref{quantities} are the
solutions of the following linear recursion relation, for $\al\geq 1$:
\begin{equation}\label{recuhocor}
C_{\al+1,m}(j,k)=C_{\al,m}(j-2,k)+
y_{2\al+1}(j-\al,k)\, C_{\al,m-1}(j,k) + y_{2\al}(j-\al-1,k)\, C_{\al-1,m-1}(j-2,k)
\end{equation}
with coefficients:
\begin{eqnarray}\label{timey}
y_{2\al+1}(j,k)&=&{t_{\al+1,j+k-1,k+1}\, t_{\al,j+1+k,k}\over t_{\al+1,j+k,k}\, 
t_{\al,j+k,k+1}}\quad
(\al\geq 1)\nonumber \\
y_{2\al}(j,k)&=&{t_{\al+1,j+k,k+1}\, t_{\al-1,j+k+1,k}\over t_{\al,j+k,k} \, t_{\al,j+k
+1,k+1}}
\quad (\al\geq 1)\nonumber \\
y_1(j,k)&=&{t_{1,j+k,k+1}\over t_{1,j+k+1,k}}\ ,
\end{eqnarray}
subject to the initial conditions $C_{0,m}(j,k)=\delta_{m,0}$ and 
$C_{1,m}(j,k)=\delta_{m,0}+\delta_{m,1}\, y_1(j-1,k)$.
\end{cor}

\begin{example}\label{examC}
We have the following first few values of $C_{\al+1,m}(j,k)$:
\begin{eqnarray*}
\al=0:\ \ C_{1,0}(j,k)&=&1\\
C_{1,1}(j,k)&=&y_1(j-1,k)\\
\al=1:\ \ C_{2,0}(j,k)&=&1\\
C_{2,1}(j,k)&=&y_1(j-3,k)+y_2(j-2,k)+y_3(j-1,k)\\
C_{2,2}(j,k)&=&y_1(j-1,k)y_3(j-1,k)\\
\al=2:\ \ C_{3,0}(j,k)&=&1\\
C_{3,1}(j,k)&=&y_1(j-5,k)+y_2(j-4,k)+y_3(j-3,k)+y_4(j-3,k)+y_5(j-2,k) \\
C_{3,2}(j,k)&=&y_1(j-3,k)y_3(j-3,k)+y_4(j-3,k)y_1(j-3,k)\\
&&+y_5(j-2,k)(y_1(j-3,k)+y_2(j-2,k)+y_3(j-1,k))\\
C_{3,3}(j,k)&=&y_1(j-1,k) y_3(j-1,k) y_5(j-2,k) 
\end{eqnarray*}
\end{example}

\begin{remark}
The Corollary \ref{correcu} allows to interpret the conserved quantities of the 
$A_r$ $T$-system
as follows. From the recursion relation \eqref{recuhocor}, we deduce that 
$C_{r+1,m}(j,k)$
is a homogeneous polynomial of the weights $y_1,y_2,...,y_{2r+1}$, 
themselves ratios of
products of some $t_{a,b,c}$'s with $c$ only taking the values $k$ and $k+1$. 
If we impose
$t_{r+1,j,k}=1$, we see that, as explained above, $C_{r+1,m}(j,k)=c_m(j)$ is 
independent 
of $k$. We may therefore
write $C_{r+1,m}(j,k)=C_{r+1,m}(j,0)$, the latter involving only $T_{a,b,c}$'s 
with $c=0,1$.
These give $r$ conservation laws for $m=1,2,...,r$. For $r=1$, we have for 
instance
\begin{eqnarray*}
C_{2,1}(j,k)&=&{T_{1,j+k,k}\over T_{1,j+k-1,k+1}}+{1\over T_{1,j+k-1,k
+1}T_{1,j+k-2,k}}
+{T_{1,j+k-3,k+1}\over T_{1,j+k-2,k}}\\
&=&C_{2,1}(j,0)={T_{1,j,0}\over T_{1,j-1,1}}+{1\over T_{1,j-1,1}T_{1,j-2,0}}
+{T_{1,j-3,1}\over T_{1,j-2,0}}\\
\end{eqnarray*}
\end{remark}

\subsection{Hard particle interpretation}

In this paper, we introduce a slightly generalized model of hard particles on a 
graph. 
\subsubsection{Definition of the model}
\begin{figure}
\centering
\includegraphics[width=12.cm]{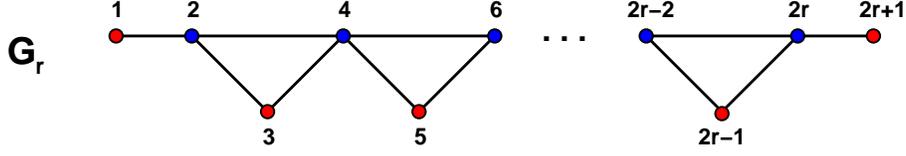}
\caption{The graph $G_r$,
with $2r+1$ vertices labeled $i=1,2,...,2r+1$.}
\label{fig:graphgr}
\end{figure}

Let $G_r$
be the graph 
of Figure \ref{fig:graphgr}, 
with vertices labeled as shown. 
When $r=1$, $G_1$ is just the chain with 3 vertices, and
when $r=0$ $G_0$ is a single vertex.

To each vertex labeled $i$ in $G_r$, we associate a height function
$h$, where 
$$
h(i)=\left\lfloor {i+1\over 2} \right\rfloor, \ (i>1), \ h(1)=0.
$$

A {\em configuration of hard particles} on $G_r$ is a subset $S$
of $I_{2r+1}$ such that $i,j\in I$ implies that vertices $i$ and $j$
are not connected by an edge. We can think of the elements of $I$ as
the vertices occupied by particles. The set of all hard particle
configurations of cardinality $m$ on $G_r$ is called $\mathcal C_m$.
There is a natural ordering on the set $I_{2r+1}$, and in the
generalized hard particle model we define in this paper, the set $S$
is considered to be an ordered set.

In general, a hard particle model on $G_r$ associates weights to the
occupied vertices which depend on the vertex label, and possibly also
on the total number of occupied particles.  The corresponding
partition function is the sum over all possible hard-particle
configurations of the products of the occupied vertex weights.

For the purpose of this work, 
we define the partition function for $m$ hard particles as
\begin{equation}\label{hardparticleZ}
Z_{m}^{G_r}(j,k)=\sum_{S\in \mathcal C_m} \prod_{\ell=1}^m
y_{i_\ell}(j-2(r+\ell-m)-1+h(i_\ell),k)
\end{equation}
with the weights $y_i$ as in \eqref{timey} and $S=\{i_1,...,i_m\}$.

\subsubsection{Conserved quantities as hard particle partition functions}

We have the following.

\begin{thm}
The partition function $Z_m^{G_\al}(j,k)$ \eqref{hardparticleZ} for $m$-hard 
particles on $G_\al$
coincides with the quantity $C_{\al+1,m}(j,k)$ of \eqref{quantities}.
\end{thm}
\begin{proof}
Hard particle partition functions on $G_r$ satisfy a recursion
relation in $r$. 
Fix $m$ and consider the configuration of particles on vertices $(2 r
+1,2r)$. There are 3 possible pairs of occupation numbers for these two 
neighboring
vertices, $(0,0)$, $(1,0)$ and $(0,1)$, respectively contributing to the partition 
function:
\begin{itemize}
\item $(0,0)$ contributes $Z_m^{G_{r-1}}(j-2,k)$.
\item $(1,0)$ contributes $y_{2 r+1}(j-r,k)\, Z_{m-1}^{G_{r-1}}(j,k)$.
\item $(0,1)$ contributes $y_{2r}(j-r-1,k)\, Z_{m-1}^{G_{r-2}}(j-2,k)$.
\end{itemize}
This implies that $Z_m^{G_{r}}$ satisfies the recursion relation
\begin{equation}\label{recuHO}
Z_m^{G_{r+1}}(j,k)=Z_m^{G_r}(j-2,k)+
y_{2r+1}(j-r,k)\, Z_{m-1}^{G_r}(j,k) + y_{2r}(j-r-1,k)\,
Z_{m-1}^{G_{r-1}}(j-2,k).
\end{equation}
But this is the same relation satisfied by $C_{r,m}(j,k)$, Equation
\eqref{recuhocor}, with the same initial conditions,
$Z_0^{G_r}(j,k)=1$ (for any $r$) and and
$Z_1^{G_0}(j,k)=y_1(j-1,k)=C_{1,1}(j,k)$. The theorem follows.
\end{proof}

Setting $t_{r+1,j,k}=1$, we have:
\begin{cor}

The conserved quantities $c_m(j)$ of the $A_r$ $T$-system are the partition 
functions
for $m$-hard particles on $G_r$, with the weights:
\begin{eqnarray}\label{timeyT}
y_{2\al+1}(j,k)&=&{T_{\al+1,j+k-1,k+1}\, T_{\al,j+1+k,k}\over T_{\al+1,j+k,k}\, 
T_{\al,j+k,k+1}}\quad
(1\leq \al\leq r)\nonumber \\
y_{2\al}(j,k)&=&{T_{\al+1,j+k,k+1}\, T_{\al-1,j+k+1,k}\over T_{\al,j+k,k} \, T_{\al,j
+k+1,k+1}}
\quad (1\leq\al\leq r)\nonumber \\
y_1(j,k)&=&{T_{1,j+k,k+1}\over T_{1,j+k+1,k}}
\end{eqnarray}
where $T_{0,j,k}=T_{r+1,j,k}=1$ for all $j,k\in\Z$.
\end{cor}

As the resulting hard-particle partition functions are independent of
$k$, we may set $k=0$ in the expression for the weights.

\subsection{A pictorial representation for the hard particle partition
  function} 

\begin{figure}
\centering \includegraphics[width=12.cm]{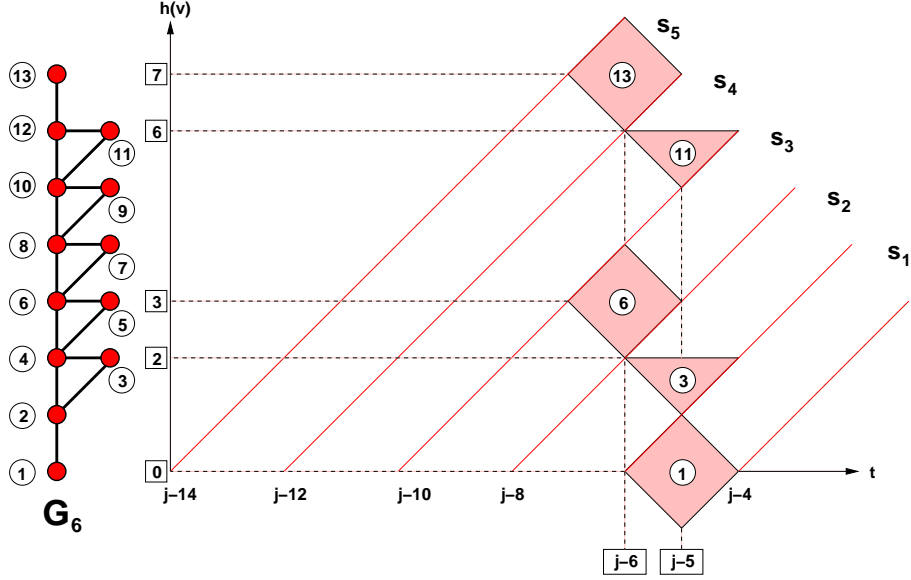}
\caption{\small A graphical interpretation of a hard-particle
configuration on the graph $G_6$, with $m=5$ particles at positions
$\{ 1,3,6,11,13\}$.  The label of
the occupied vertex is indicated in a circle, and the time and height
coordinates in rectangles. A distinct diagonal stripe corresponds to
each particle. The leftmost stripe has $x$-intercept $j-2(r+1)=j-14$,
and the rightmost is at $j-2(r+1-m)=j-4$. The weight of this
configuration is
$y_1(j-5,k)y_3(j-5,k)y_6(j-6,k)y_{11}(j-5,k)y_{13}(j-6,k)$.}
\label{fig:hardotime}
\end{figure}

The hard particle configurations which give rise to the partition
 function of the form \eqref{hardparticleZ} can be represented
 graphically as in Figure \ref{fig:hardotime}.
\begin{itemize}
\item A particle at a spine vertex $v$
 ($v\in\{1,2,4,6,...,2r-2,2r,2r+1\}$) is represented by a diamond on the
 two-dimensional lattice, its center at the height of the vertex, at
 the point $(t,h(v))$ for some $t\in\Z$, and its vertices at the four
 neighboring lattice sites.
\item A particle a vertex $v\in\{3,5,7,...,2r-1$\} is represented by the
lower half of such a diamond.
\end{itemize}

We call $t$ the time coordinate, and $h(v)$ the height.  Each polygon
is at $(t,h(v))$ contained in a diagonal stripe $s$, bordered by
the lines $y=x-(t+1-h(v))$ and $y=x-(t-1-h(v))$. We denote $s$
by its $x$-intercepts, $s=\{t-1-h(v),t+1-h(v)\}$.

Given a configuration $S\in\cC_m$, with $S=\{i_1<i_2< \cdots <i_m\}$,
 the polygon representing the particle $i_1$ is drawn in the stripe
 $s_1=\{t-2,t\}$; that of $i_2$ in the stripe immediately above and to
 the left, $s_2=\{t-4,t-2\}$, and the $k$-th polygon representing
 $i_k$ lies in stripe $s_k=\{t-2k,t-2k+2\}$. The height of each
 polygon is determined by $h(i_j)$ and its time coordinate by its
 stripe:  $t_k=h(i_k)-1+t-2(k-1)$.

If we choose $t=j-2(r+1-m)$, then Equation \eqref{hardparticleZ}
 can be written as
\begin{equation}\label{goodZ}
Z_{m}^{G_r}(j,k)=\sum_{S\in\mathcal C_m}
\prod_{\ell=1}^m y_{i_\ell}(t_\ell,k)
\end{equation}

\section{Path formulation and positivity}\label{firpositsec}

We now give an expression for $T_{\al,j,k}$ as a function of the
initial data $\bx_0=(T_{\beta,j,0},T_{\beta,j,1})_{\beta\in I_r,j\in
\Z}$. It can be interpreted as the partition function of weighted
paths on a certain graph, with time-dependent weights. That is, we
generalize the notion of a weighted path, so that the weight of a step
in the path depends on the time at which it is taken.

As a corollary of the formulation in this section, we have the
positivity Theorem \ref{positheor} for the variables $T_{\al,j,k}$ as
a function of the initial data.

\subsection{Definitions}\label{pathdefs}
\label{transmatsec}

\begin{figure}
\centering
\includegraphics[width=12.cm]{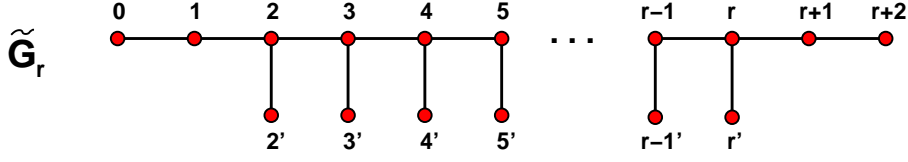}
\caption{The graph ${\wG}_r$,
with $2r+2$ vertices.}
\label{fig:dualgr}
\end{figure}

Let ${\wG}_r$ be the graph in Figure \ref{fig:dualgr}. It has $2r+2$
vertices, which are ordered as $(0,1,2,2',3,3',...r,r',r+1,r+2)$. 
Its incidence matrix $A$ is
\begin{eqnarray*}
&& A_{m,m'}=A_{m',m}=1, \ (2\leq m\leq r);\qquad
 A_{m,m+1}=A_{m+1,m}=1, \ (0\leq m\leq r+1).
\end{eqnarray*}
The vertex labelled $0$ is called the origin of the graph. We call the
vertices $i$ the {\em spine} vertices of $\wG_r$, and the edges which
connect $i\to\i\pm1$ spine edges.

We consider the set ${\mathcal P}_{t_1,t_2}^{a,b}$ of paths $p$ on the
graph ${\tilde G}_r$, starting at time $t_1$ and vertex $a$, and
ending at time $t_2\geq t_1$ at vertex $b$. We take $t_i\in \Z$, and
each step takes one time unit.  The path $p$ may be represented by the
succession of visited vertices, $p=(p(t))_{t=t_1,t_1+1,...,t_2}$, with
$p(t_1)=a$ and $p(t_2)=b$ and
$A_{p(s),p(s+1)}=1$ for any $s$.

Let $w_{i,j}(t)$ be the weight of a step vertex $i$ to vertex $j$ at
time $t$. We define the weight of a path $p\in {\mathcal
  P}_{t_1,t_2}^{a,b}$ to be
\begin{equation}
w(p)=\prod_{s=t_1}^{t_2-1} w_{p(s),p(s+1)}(s),\quad p(t_1)=a,\ p(t_2)=b.
\end{equation}
The partition function for weighted paths in ${\mathcal
P}_{t_1,t_2}^{a,b}$ is
\begin{equation}
\cZ_{t_1,t_2}^{a,b}=\sum_{p\in {\mathcal P}_{t_1,t_2}^{a,b}} w(p).
\end{equation}
For later use, we define $\cZ_{t_1,t_2}^{a,b}=0$ if $t_1>t_2$.

\begin{figure}
\centering
\includegraphics[width=10.cm]{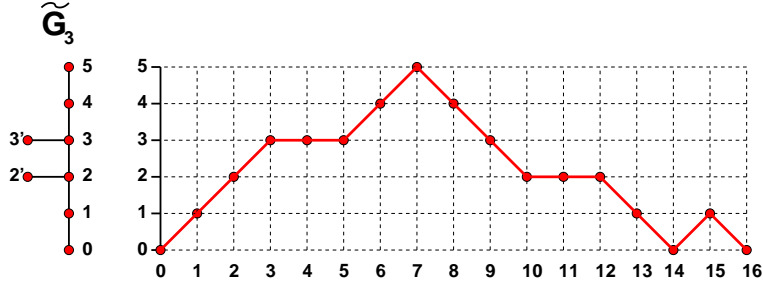}
\caption{\small The planar representation of a typical path in
  $\mathcal P_{0,16}^{0,0}$ on the
  graph ${\tilde G}_3$. }
\label{fig:grpath}
\end{figure}

Paths can be represented on the
lattice $\Z^2$ as in Figure \ref{fig:grpath}.  We associate a vertical
coordinate $h(i)=h(i')=i$ to each vertex of ${\wG}_r$.  The horizontal
axis is the time.
A step $a\to b$ at time $t$ on
${\wG}_r$ is a step $(t,h(a))\to (t+1,h(b))$. 

We claim (see Theorem \ref{pathrepsthm}) that there exists a choice of
weights $w_{a,b}(s)$, as functions of
$\bx_0=(T_{\al,j,0},T_{\al,j,1})_{\al\in I_r,j\in \Z}$, such that
$T_{1,j,k}/T_{1,j+k,0}$ is equal to the partition function
$\cZ_{j-k,j+k}^{0,0}$.

Dividing a path, which takes place from time $t$ to time $t'$,
into a first part from $t$ to $t'$, and a
second part, from $t'$ to $t''$, we have
\begin{equation}\label{recuZ}
\cZ_{t,t'}^{a,b}=\sum_{x\in \wG_r} \cZ_{t,t'}^{a,x}\,
\cZ_{t',t''}^{x,b},\quad t'\in[t,t''].
\end{equation}
In particular, the matrix of one-step partition functions
$\cZ_{t,t+1}^{a,b}$ is called the transfer matrix $\cT(t)$, with entries
\begin{equation}\label{transmat}
(\cT(t))_{a,b}=\cZ_{t,t+1}^{a,b}=w_{a,b}(t) \, A_{a,b}.
\end{equation}
The transfer matrix is a decorated adjacency matrix.  The recursion
relation \eqref{recuZ} implies
\begin{equation}\label{recunZ}
\cZ_{t_1,t_2}^{a,b}
=\left( \cT(t_1)\cT(t_1+1)\cdots \cT(t_2-1)\right)_{a,b}.
\end{equation}

We use the following definition for weights $w_{a,b}(s)$ of paths on $\wG_r$:
\begin{eqnarray}\label{choicew}
&& w_{m,m'}(s)=1, \quad w_{m',m}(s)=y_{2m+1}(s,0), \ (m\in
\{1,...,r\}), \nonumber \\ && w_{m,m+1}(s)=1, \quad
w_{m+1,m}(s)=y_{2m}(s,0), \ (m\in\{1,...,r-1\}), \\ && w_{0,1}(s)=1,
\quad w_{1,0}(s)=y_1(s,0), \nonumber
\end{eqnarray}
in terms of the weights $y_i(s,k)$ of Equation \eqref{timeyT}.

\subsection{An involution on pairs of weights}
 We define an involution
$\varphi$ on the set $\cC_m\times {\cP}_{j-2(r+1-m),j+2k}^{0,0}$,
consisting of hard-particle configurations on $G_r$ and paths on
$\wG_r$.

Let $(S,p)\in \cC_m\times {\cP}_{j-2(r+1-m),j+2k}^{0,0}$, with
$m\in\{0,...,r+1\}$, $k\in\Z_+$.  We refer to the graphical
representations of Figures \ref{fig:hardotime} and \ref{fig:grpath},
and we draw $S$ and $p$ on the same lattice (see Figure
\ref{fig:invol}), where $S$ is represented between the diagonal lines
$y=x-(j-2(r+1))$ and $y=x-(j-2(r+1-m))$, and $p$ starts at
$(j-2(r+1-m),0)$, the $x$-intercept of the bottom stripe of $S$.

\begin{figure}
\centering
\includegraphics[width=15.cm]{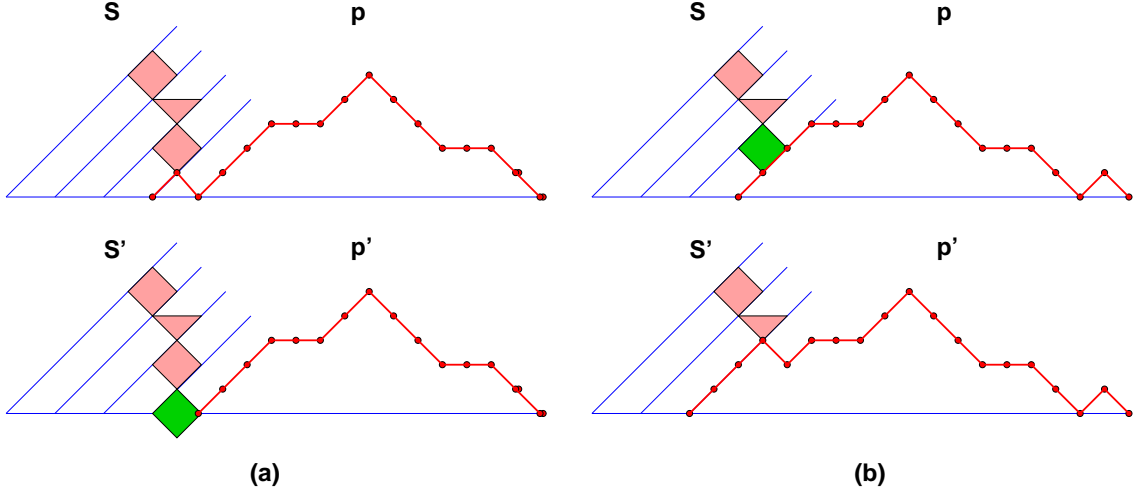}
\caption{\small The involution $\varphi$ between pairs of $m$- hard
particle configurations and paths in
${\mathcal P}_{j-2(r+1-m),j+2k}^{0,0}$. Case (a): the first stripe
traversed by
$p$ is absorbed into $S'$, which has $m+1$ particles, and 
$p'\in{\mathcal P}_{j-2(r+1-m)+2,j+2k}^{0,0}$. Case
(b): The bottom stripe of $S$ is absorbed into $p'$, now in ${\mathcal
P}_{j-2(r+1-m)-2,j+2k}^{0,0}$, while $S'$ has only $m-1$ 
particles.}
\label{fig:invol}
\end{figure}

The path $p$ has an initial section $p_0$ within the diagonal stripe
$\{j-2(r+1-m),j-2(r-m)\}$, consisting of $u$ consecutive up steps and
(i) a down step $(s,u)\to (s+1,u-1)$ or (ii) two horizontal steps
$(s,u)\to(s+1,u)\to (s+2,u)$, where $s=j-2(r+1-m)+u$. $p\setminus p_0$
is then to the right of this initial stripe.

Let $\sigma$ be a map from path steps of type (i) or (ii) on $\wG_r$ to the
vertex set of $G_r$. It is defined as follows:
\begin{eqnarray*}
\sigma((s,u)\to(s+1,u-1))&=&\left\{\begin{array}{ll}
2u-2, & 2\leq u\leq r+1;\\
1, & u=1;\\
2r+1, & u=r+2.\end{array}\right.\\
\sigma((s,u)\to(s+2,u)) &= & 2u-1.
\end{eqnarray*}

\begin{remark}
Graphically, steps of type (i) and (ii) in $p$ can mapped precisely to
the polygons representing particles on $G_r$. A step of type (i) is
the NE edge of a diamond (hence a particle on a spine vertex) and a
step of type (ii) the upper edge of a half-diamond. The map $\sigma$
represents this correspondence.
\end{remark}

Denote by $i$ the image of a step under the map $\sigma$. We must now
distinguish between two cases.
\begin{itemize}
\item {\bf Case (a):} If $i<i_1$ and $S':=\{i,i_1,i_2,...,i_m\}\subset
  \cC_{m+1}$, define $p'\in {\mathcal P}_{j-2(r+1-m)+2,j+2k}^{0,0}$ to
  be the path with $p'(j-2(r+1-m)+2+x)=x$ for $x=0,1,...,u$ (case (i))
  or $x=0,1,...,u-1$ (case (ii)), and $p'(x)=p(x+2)$ otherwise.
\item {\bf Case (b):} If $i\geq i_1$ or $\{i,i_1,...,i_m\} \notin\cC_{m+1}$,
define $S'=\{i_2,i_3,...,i_m\}\in\cC_{m-1}$, the hard particle
configuration with the right stripe removed. It is now drawn between
the diagonal lines
$y=x-(j-2(r+1))$ and $y=x-(j-2(r+1-(m-1)))$. As for the path $p'$,
\begin{itemize}
\item If $i_1\in \{3,5,...,2r-1\}$, define $p'(j-2(r+2-m)+x)=x$ for
  $x=0,1,...,h(i_1)$, 
$p'(j-2(r+2-m)+h(i_1)+y)=h(i_1)$ 
for $y=1,2$.
\item Otherwise, $p'(j-2(r+2-m)+x)=x$
for $x=0,1,...,h(i_1)+1$, and $p'(j-2(r+2-m)+h(i_1)+2)=h(i_1)$.
\end{itemize}
In both cases, $p'(x)=p(x-2)$ for the remaining times.
\end{itemize}

\begin{remark} Graphically the map can be visualized as follows. If
  the particle represented by $p_0$ can be added to $S$ while keeping
  the hard-particle condition, then we do this, while changing $p_0$
  so that it consists only of up steps, starting two steps to the
  right of the original starting point of $p$. Otherwise, perform the
  opposite operation, changing the first particle to a path segment.
\end{remark}

In view of the graphical description, the map $\varphi$ is clearly an
involution. Moreover it is weight-preserving: In Equation
\eqref{choicew}, only the steps of type (i) or (ii) have a non-trivial
weight. Moreover, $w(\sigma({\rm step}))=w({\rm step})$ according to
Equation \eqref{timeyT} (setting $k=0$).  Therefore,
$w(S,p)=w(S)w(p)=w(S')w(p')=w(S',p').$ We have
\begin{lemma}
\begin{equation}
\sum_{m=0}^{r+1} (-1)^{r+1-m} Z_{m}^{G_r}(j,0)\,
\cZ_{j-2(r+1-m),j+2k}^{0,0}=0 \label{ornotobe}, \quad (j\in \Z, k\geq 0).
\end{equation}
\end{lemma}
\begin{proof}
This is the partition function for pairs $(S,p)\in
\cup_{m=0}^{r+1}\cC_m\times {\cP}_{j-2(r+1-m),j+2k}^{0,0}$, with an extra 
factor
$(-1)^{r+1-m}$ which ensures  that the
contributions of $(S,p)$ and $\varphi(S,p)$ cancel each other.
\end{proof}

We can also consider the sum in Equation \eqref{ornotobe} in the case
where $k<0$. The sum is non-trivial in those cases only if $k\geq
-r-1$, since $\cZ_{t,t'}^{a,b}=0$ if $t>t'$. We extend the definition
of $\varphi$: $\varphi(S,p)=(S,p)$ if $S=\emptyset$ or if
the path $p$ has length zero and $i_1>1$.

\begin{lemma}
\begin{equation}\label{initzp}
\sum_{m=0}^{r+1-i} (-1)^{r+1-m} Z_{m}^{G_r}(j,0)\,
\cZ_{j-2(r+1-m),j-2i}^{0,0}=(-1)^i Z_{r+1-i}^{G_r'}(j,0)
\end{equation}
where $Z_m^{G_r'}(j,k)$ is the partition function of $m$ hard
particles on $G_r'$, the graph $G_r$ with vertex $1$ removed (or the
contribution to $Z_m^{G_r}$ in which vertex $1$ is unoccupied).
\end{lemma}
\begin{proof}
We apply the involution argument in the previous Lemma for $k$ in the
range $-r-1\leq
k < 0$.  Pairs $(S,p)$
which are not invariant under $\varphi$ cancel each other. We are left
with the contribution of the invariant pairs. The latter always have
$p=\emptyset$ and the vertex $1$ unoccupied.  
\end{proof}

Equations \eqref{initzp} are an expression for the initial conditions
of the partition functions $\cZ_{j-2(r+1-m),j-2i}^{0,0}$ with $1\leq
i\leq r-1$ in terms of hard-particle partition functions.

\subsection{The $T$-system solution $T_{1,j,k}$ as a partition
  function of paths}
Our main result in this section is the following.
\begin{thm}\label{pathrepsthm} 
\begin{equation}\label{truc}
T_{1,j,k}=T_{1,j+k,0}\, \cZ_{j-k,j+k}^{0,0} \ .
\end{equation}
\end{thm}
\begin{proof} We will show that $S_{j,k}=T_{1,j+k,0}\, \cZ_{j-k,j+k}^{0,0}$
satisfies the linear recursion relation \eqref{linrecone} and
coincides with $T_{1,j,k}$ when $k\in\{0,\ldots, r\}$ for any $j\in
\Z$. Given that \eqref{linrecone} has $r+1$ terms, this implies
$S_{j,k}=T_{1,j,k}$ for all other $k$. 

The sum
$$
\sum_{m=0}^{r+1} (-1)^m c_{r+1-m}(j-k)\,
\cZ_{j-k-2m,j+k}^{0,0} 
$$ is equal to the sum in Equation \eqref{ornotobe}, 
since
$Z_{m}^{G_r}(j,0)=c_{m}(j)$. Therefore, it vanishes for all 
$j\in \Z$ and $k\in \Z_+$. This implies that $S_{j,k}$ satisfies the
same recursion relation \eqref{linrecone} as $T_{1,j,k}$.

As for the initial conditions, we see from Equation \eqref{initzp} that
$S_{j,k}$ satisfies
\begin{equation}\label{initzS}
\sum_{m=0}^{i}  (-1)^{m} Z_{i-m}^{G_r}(j,0)\, S_{j-m-2(r+1-i),m}
=Z_{i}^{G_r'}(j,0)\, T_{1,j-2(r+1-i),0}.
\end{equation}

We will show that the variables $T_{1,j,k}$ satisfy the same relations.
Let
\begin{equation}
W_i^{G_\al}(j)=\sum_{m=0}^{i}  (-1)^{m} Z_{i-m}^{G_\al}(j,0)\, 
{T_{1,j-2(\al+1-i)-m,m}\over T_{1,j-2(\al+1-i),0}}, \ \ (0\leq i\leq \al+1)
\end{equation}
Using the recursion relations \eqref{recuHO}, we find
\begin{eqnarray*}
W_i^{G_{\al+1}}(j)=
W_{i}^{G_{\al}}(j-2)+y_{2\al+1}(j-\al)W_{i-1}^{G_{\al}}(j)+
y_{2\al}(j-\al-1)W_{i-1}^{G_{\al-1}}(j-2).
\end{eqnarray*}
This is identical to the recursion relations \eqref{recuHO} for
$Z_i^{G_\al}(j,0)$. 
Comparing the initial terms, we find that $W_0^{G_0}(j)=1$ and $W_1^{G_0}
(j)=
Z_1^{G_0}(j,0)-{T_{1,j-1,1}\over T_{1,j,0}}Z_{0}^{G_0}(j,0)=0$, so that 
$W_i^{G_0}(j)=Z_i^{G_0}(j,0)\vert_{y_1(j-1)=0}$. Moreover,
\begin{eqnarray*}
W_0^{G_1}(j)&=&1,\\
W_1^{G_1}(j)&=&Z_1^{G_1}(j,0)-{T_{1,j-3,1}\over T_{1,j-2,0}}Z_{0}^{G_0}(j,
0)=Z_1^{G_1}(j,0)\vert_{y_1(j-3)=0},\\
W_2^{G_1}(j)&=&Z_2^{G_1}(j,0)-{T_{1,j-1,1}\over T_{1,j,0}}Z_{1}^{G_1}(j,0)+
{T_{1,j-2,2}\over T_{1,j,0}}Z_{0}^{G_1}(j,0)=0,
\end{eqnarray*}
where we have used the identity between $Z$ and $C$, Example
\ref{examC}, the definitions \eqref{timeyT}, and
$T$-system relations. In short, we have
$W_i^{G_\al}(j)=Z_i^{G_\al}(j,0)\vert_{y_1=0}$, valid for all 
initial data $\al=0,1$. This implies
\begin{equation}
W_i^{G_\al}(j)=Z_i^{G_\al}(j,0)\vert_{y_1=0} \quad {\rm for}\ \ {\rm
  all}\ \ \al .
\end{equation}
Thus, $W_i^{G_\al}(j)$ is the partition function for $i$ hard
particles on $G_\al$ with weight $0$ on vertex $1$, or
alternatively, vertex $1$ unoccupied. Therefore,
$W_i^{G_r}(j)=Z_i^{G_r'}(j,0)$. 

Therefore, $T_{1,j,k}$ and $S_{j,k}$ satisfy the
same recursion relation and have the same boundary
conditions. The Theorem follows.
\end{proof}

The reasoning of Theorem \ref{pathrepsthm} can be carried through by
considering paths with weights $y_\al(j)$ replaced by the weights
$y_\al(j,k)$ of eq.\eqref{timey}.  We therefore have the following.

\begin{cor}
If we define the weights $w_{a,b}(s,p)$ as follows:
\begin{eqnarray}
&& w_{m,m'}(s,p)=1, \quad w_{m',m}(s,p)=y_{2m+1}(s,p), \ (m\in \{1,...,r\}), 
\nonumber \\
&& w_{m,m+1}(s,p)=1, \quad w_{m+1,m}(s,p)=y_{2m}(s,p), \ (m\in\{1,...,r-1\}),
\nonumber  \\
&& w_{0,1}(s,p)=1, \quad w_{1,0}(s,p)=y_1(s,p), \nonumber 
\end{eqnarray}
with the $y_i(s,p)$ as in \eqref{timeyT}, then the following identity holds:
\begin{equation}
T_{1,j,k}=T_{1,j+p,k-p}\, Z_{j-p,j+p}^{0,0}(\{y_\al(s,k-p), \ j-p\leq s\leq j+p, 1\leq 
\al\leq 2r+1\}) \ .
\end{equation}
\end{cor}
\begin{proof}
We may view this as a particular case of the translational invariance
$T_{\al,j,k}\to T_{\al,j,k+1}$ of the $T$-system, namely that
$T_{\al,j,k}$ is expressed as the {\it same} function of the initial
data $\{T_{\beta,j,k-p},T_{\beta,j,k-p+1}\}_{\beta\in I_r,j\in \Z}$ as
$T_{\al,j,p}$ is expressed in terms of
$\{T_{\beta,j,0},T_{\beta,j,1}\}_{\beta\in I_r,j\in \Z}$.  We deduce
that $T_{1,j,p}=T_{1,j+p,0}Z_{j-p,j+p}^{0,0}(\{y_\al(s,0)\})$ has the
same expression as $T_{1,j,k}=T_{1,j,p+k-p}$ in terms of
$y_\al(s,k-p)$, and the corollary follows, as the prefactor itself
comes from the substitution $T_{1,j+p,0}\to T_{1,j+p,k-p}$.
\end{proof}

\subsection{General $T$-system solution $T_{\al,j,k}$: families of non-
intersecting paths}\label{noninter}

We may interpret $T_{\al,j,k}$ directly in terms of paths
by use of the determinant expression of Theorem \ref{elim} for $T_{\al,j,k}$ in
terms of the $T_{1,\ell,m}$. Indeed, given weighted paths on an acyclic graph 
$\Gamma$, say with
partition function ${\mathcal Z}_{s,e}$ for paths starting at vertex $s$ and 
ending at vertex $e$,
the Lindstr\"om-Gessel-Viennot formula
gives an expression for the partition function of $\al$ non-intersecting paths on 
$\Gamma$ (i.e. such that
no to paths share a vertex) as ${\mathcal Z}_{s_1,...,s_\al;e_1,...,e_\al}=
\det_{1\leq i,j\leq \al}{\mathcal Z}_{s_i,e_j}$.
We obtain:
\begin{eqnarray}
T_{\al,j,k}&=&\det_{1\leq a,b\leq \al} \left(T_{1,j+k+2b-\al-1,0} \, 
{\mathcal Z}_{j-k+\al+1-2a,j+k+2b-\al-1}^{0,0}\right) \nonumber \\
&=&\left(\prod_{b=1}^\al T_{1,j+k+2b-\al-1,0}\right) \, {\mathcal Z}_{s_1,...,s_
\al;e_1,...,e_\al}^{0,0}\label{lgvalpha}
\end{eqnarray}
where ${\mathcal Z}_{s_1,...,s_\al;e_1,...,e_\al}^{0,0}$ stands for the partition 
function of families
of $\al$ non-intersecting paths in the plane representation of Section 
\ref{pathdefs}, 
starting at the points $s_a=(j-k+2a-\al-1,0)$, $a=1,2,...,\al$ and ending at the 
points $e_b=(j+k+\al+1-2b,0)$, $b=1,2,...,\al$. 
Alternatively, one may think of the partition function ${\mathcal Z}_{s_1,...,s_
\al;e_1,...,e_\al}$
as that of $\al$ ``vicious" walkers (i.e. never meeting at a vertex) on ${\tilde 
G}_r$, going
from the root to the root, respectively starting at times $j-k+2a-\al-1$ 
and ending at times $j+k+\al+1-2a$, $a=1,2,...,\al$, each step corresponding 
to a unit of time.

Interpreted in this way, the $T_{\al,j,k}$ are manifestly positive Laurent 
polynomials of the initial data,
via the weights $y_\beta(t)$ and the prefactor in \eqref{lgvalpha}. We therefore 
have the:

\begin{thm}\label{positheor}
The solution $T_{\al,j,k}$ of the $T$-system is expressed as a
positive Laurent polynomial of the initial data
$\bx_0=\{T_{\beta,j,0},T_{\beta,j,1}\}_{\beta\in I_r,j\in \Z}$ for all
$\al\in I_r$ and all $j,k\in \Z$.
\end{thm}
\begin{proof}
The statement is clear from the above discussion for $k\geq \al+1$ for
which all the partition functions in the determinant \eqref{lgvalpha}
have the form ${\mathcal Z}^{0,0}_{t,u}$ with $t\leq u$, and therefore
can be interpreted within the LGV framework. For $k\geq 0$ however,
from the structure of the $T$-system, it is clear that $T_{\al,j,k}$
only depends on a finite part of the initial data
$\{T_{\beta,\ell,0},T_{\beta,\ell,1},\ \vert \beta-\al \vert <k,
\vert\ell-j\vert <k\}$. In particular, if $0<k<\al+1$, then as only
the $\beta\geq \al+1-k>0$ are involved, we may truncate the size of
the $T$-system to some $A_{r'}$, with $r'=r-(\al+1-k)<r$. Upon
renaming the initial data accordingly, we may interpret $T_{\al,j,k}$
as $T_{\al',j,k}$ in this new $T$-system, where
$\al'=\al-(\al+1-k)=k-1$. For this $\al'\geq k-1$ the LGV formula
applies, and positivity follows. Finally, note that the expression of
$T_{\al,j,1-k}$ in terms of $\{T_{\beta,j,0},T_{\beta,j,1}\}$ is the
same as that of $T_{\al,j,k}$ in terms of the reflected initial data
$\{T_{\beta,j,1},T_{\beta,j,0}\}$, hence positivity follows for $k<0$
as well.
\end{proof}

\section{Operator formulation and positivity in terms of mutated initial data}\label{sec:otherpaths}
Let $\mathcal A$ be the space of Laurent polynomials in the variables
$\{T_{\al,j,k}\}$.  We consider the invertible ``shift operator'' $d$
acting on the infinite-dimensional vector space over $\mathcal A$ with
basis $\{|t\rangle: t\in \Z\}$, with $d |t\rangle = |t-1\rangle$. It
acts on the restricted dual space $V^*$, with basis $\langle t |$ such
that $\langle t|t'\rangle = \delta_{t,t'}$, as $\langle t | d =
\langle t+1|.$ We consider the algebra of formal Laurent series in $d$
with coefficients in $\mathcal A$ acting on $V$. All operator
relations which we derive below are considered in the weak sense, as
identities between matrix elements. We also adopt the operator
notation for diagonal operators in this basis, for example, 
$w_{a,b}|t\rangle=w_{a,b}(t)|t\rangle$.

\subsection{An expression using operator continued fractions}\label{pathcfrac}

Theorem \ref{pathrepsthm} implies
\begin{equation}
T_{1,j,k}=T_{1,j+k,0}\, \left( \cT(j-k)\cT(j-k+1)\cdots
\cT(j+k-1)\right)_{0,0}
\end{equation}
where the transfer matrix $\cT(s)$ 
is defined in Equation \eqref{choicew}.

We define the operator-valued transfer matrix $\bbT$ to be the matrix
with entries $\langle t|\bbT_{a,b}=\cT_{a,b}(t)\langle t+1 |$. We also
define operator-valued weights $\bbY_\al$, such that 
\begin{equation}\label{skelweight}
\langle
t|\bbY_{\al}  =y_\al(t,0)\langle t+1|
\end{equation}
where the $y$'s are defined in \eqref{timeyT}. 

Using these, we can write 
\begin{equation}
T_{1,j,k}=T_{1,j+k,0}\, \langle j-k\vert\left( \bbT
^{2k}\right)_{0,0}\vert j+k\rangle =T_{1,j+k,0}\, \langle
j-k\vert\left( (I-\bbT)^{-1}\right)_{0,0}\vert j+k\rangle,
\end{equation}
where $(I-\bbT)^{-1}:=\sum_{n\geq 0} (\bbT)^n$.
Therefore, the operator $\bbF=\Big((I-\bbT)^{-1}\Big)_{0,0}$ generates
the variables $T_{1,j,k}$. 

To compute $\bbF$, we row-reduce the matrix $I-\bbT$.  The result
can be written as
a non-commutative continued fraction:
\begin{eqnarray}\label{opcont}
&& \hskip.5in\bbF= \\
&&\hskip-.3in\left( 1-d \left( 1-d \left( 1-d \bbY_3 -d \left( \cdots
(1-d\bbY_{2r-1}-d(1-d \bbY_{2r+1})^{-1}  \bbY_{2r}
)^{-1} \cdots \right)^{-1} \bbY_4 \right)^{-1} \bbY_2 \right)^{-1}
\bbY_1 \right)^{-1}. \nonumber
\end{eqnarray}
Alternatively, we can write $\bbF=\bbF_0$ where the operators $\bbF_i$
are defined inductively:
\begin{eqnarray*}
\bbF_{r+2}&=&0,\\
\bbF_k&=&\left (1 -d \bbY_{2k-1}  -d \bbF_{k+1} \bbY_{2k}
\right)^{-1},\quad  (k=r+1,r,...,3,2),\\ 
\bbF_1&=& ( 1-d \bbF_2 \bbY_2 )^{-1},\quad
\bbF_0= ( 1-d \bbF_1  \bbY_1 )^{-1}=\bbF,
\end{eqnarray*}
where each term is understood
formal power series in $d$.

\begin{figure}
\centering
\includegraphics[width=13.cm]{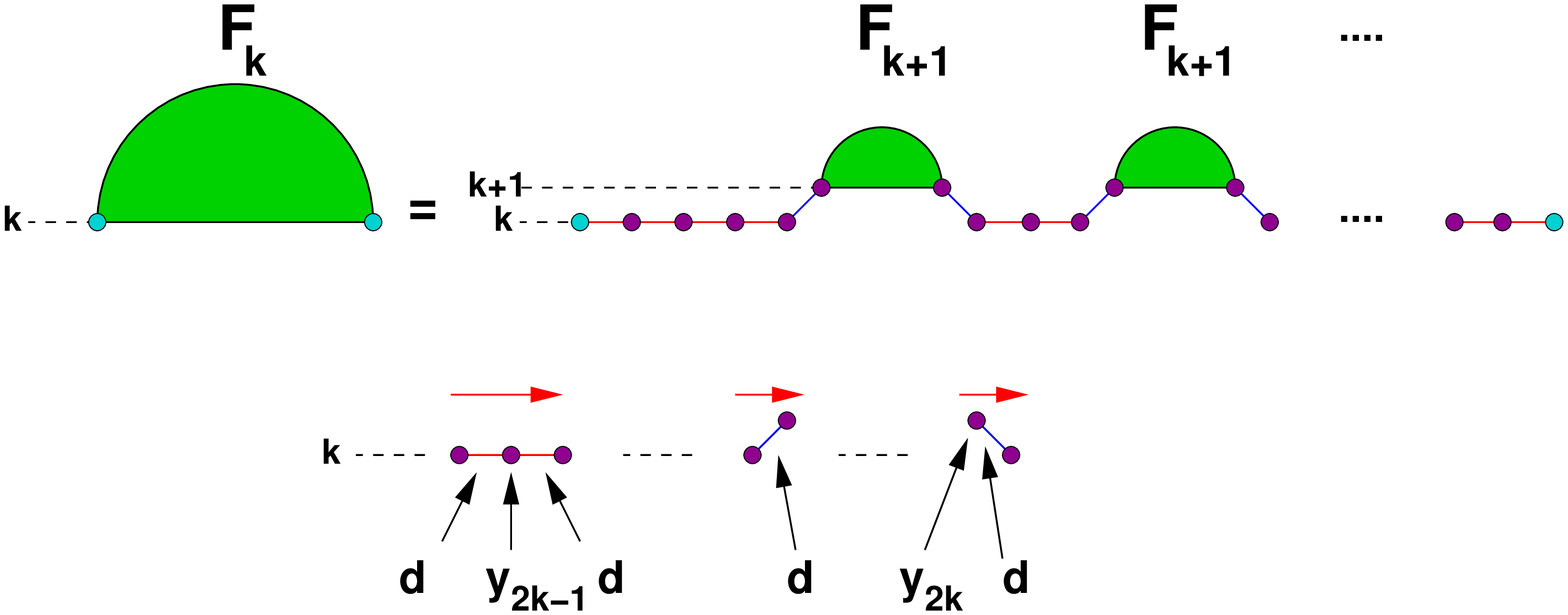}
\caption{\small The enumeration of paths on ${\wG}_r$ with
time-dependent weights $y_\al(t)$.  The paths from and to height $k$
which stay above height $k$ (generated by $\bbF_k$) are related to
those from height $k+1$ by arranging any number of horizontal step-pairs $k
\to k\to k$
and up-down step-pairs $k\to k+1\to k$, in-between which we insert any path 
from and
to height $k+1$ staying above height $k+1$.  The operator weights are
indicated on the bottom.}
\label{fig:opacityofdope}
\end{figure}

This expression is easily understood in terms of paths. Note
that each time increment corresponds to an insertion of an operator $d$.
An up step at height $k$ followed by a down
step contributes a weight $d \bbY_{2k} $, while a level step at height
$k$ contributes the weight $d \bbY_{2k-1} $. 

The operator
generating function for paths above height $k$, $\bbF_k$, is obtained by
shuffling the two following possibilities: (i) a level step pair $k\to k \to k$ (ii)
insertion of a path above height $k+1$ between steps $k\to k+1$ and $k+1\to 
k$ 
(see Figure \ref{fig:opacityofdope}).

\subsection{Mutations and operator continued fractions}
As with the $Q$-system of \cite{DFK3}, we would like to have
expressions for $T_{\al,j,k}$ as functions of other possible initial
data of cluster seeds in the cluster algebra. Cluster positivity means
that they are positive Laurent polynomials in this data, and we can
prove this by giving path generating functions on graphs with positive
weights for them. The operator formulation introduced above was
designed to allow us to do this in the case of special seeds of the form
\begin{equation}\label{xM}
\bx_{\bM}=\{T_{\al,j,m_{\al}+i}|i=0,1, \al\in I_r, j\in \Z\},
\end{equation}
where
$\bM$ is a Motzkin path of length $r$: $\bM = (m_1,...,m_r)$ with
$|m_i-m_{i+1}|\leq 1$.

This case is special, because it the non-commutative version of our
construction in \cite{DFK3} for the $Q$-system. The only difference is
that we must now use the operator-valued transfer matrix, instead of a
scalar, to account for time-dependent weights.

\subsubsection{Compound mutations and restricted initial
  data}\label{restrictinit}  The cluster seeds in
Equation \eqref{xM} are obtained from $\bx_0$ by acting on it with a
sequence of the compound mutations of the form
\begin{equation}\label{mutal}
\mu_\al = \prod_{j\in \Z} \mu_{\al,j},\qquad 
\mu_{\overline{\al}} = \prod_{j\in \Z} \mu_{\overline{\al},j}.
\end{equation}
Note that the mutation matrix $B_{0}$ has the property that
$B_{\al,\al}^{j,j'}=0$ if $j\neq j'$, hence $\mu_{\al,j}$ commutes
with $\mu_{\al,j'}$, so the compound mutations are well-defined.

The mutations \eqref{mutal} 
act on initial data $\bx_0$ via the {simultaneous} use of
all relations \eqref{Tsys} for all $j\in \Z$ to transform
$T_{\al,j;k-1}\to T_{\al,j;k+1}$ (forward mutation) or $T_{\al,j;k+1}
\to T_{\al,j;k-1}$ (backward mutation), the action being that of
$\mu_\al$ when $k$ is odd and $\mu_{\overline{\al}}$ when $k$ is
even.
Starting from the seed $\bx_0$, and acting only with
\eqref{mutal} generates a restricted set of cluster seeds. If,
moreover, we require that each of the mutations be one of the
$T$-system equations, we obtain only seeds of the form $\bx_\bM$ as in
\eqref{xM}. 

\begin{remark}
This is very similar to the situation of
Reference \cite{DFK3}, where seeds of type $\bx_\bM$ consist of variables
$\{R_{\al,m_\al}; R_{\al,m_\al+1}\}$. Here, we replace each variable
  $R_{\al, m}$ with the infinite sequence $(T_{\al,j,m})_{j\in\Z}$.
\end{remark}

\subsubsection{Operator continued fraction rearrangements}

In \cite{DFK3}, we have shown that the generating function for
$R_{1,n}$ may be expressed in terms of any mutated seed $\bx_\bM$ via
local rearrangements of the initial continued fraction in terms of the
seed $\bx_0$.  Here, we give the non-commutative version of the
starting point, which is operator version of the two rearrangement
lemmas for fractions used in \cite{DFK3}.

The following are operator identities to be understood as identities
between matrix elements of Laurent series in $d$. They are proved by a
simple calculation.

\begin{lemma}\label{rerootlemma}
Let $\bbA,\bbB$ be elements in $\mathcal A((d))$. Then
\begin{equation}\label{reroot}
1+d (1-\bbA d-d \bbB)^{-1} \bbA= \big( 1-d(1-d\bbB)^{-1}\bbA\big)^{-1}
\end{equation}
\end{lemma}

\begin{lemma}\label{reargtlemma}
Let $\bbA,\bbB,\bbC,\bbU$ be elements in $\mathcal A((d))$, with $\bbA
+ \bbB$ invertible. Then
\begin{equation}\label{reargt}
\bbA+(1-d (1-\bbU)^{-1}\bbC)^{-1} \bbB= \big( 1-\bbU-d(1-d\bbC')^{-1}\bbB' 
\big)^{-1}\bbA'
\end{equation}
where
\begin{equation}\label{abcprime}
\bbA'=\bbA+\bbB, \quad \bbB'=\bbC \bbB (\bbA+\bbB)^{-1}, \quad 
\bbC'=d^{-1}\bbC \bbA(\bbA+\bbB)^{-1} d
\end{equation}
\end{lemma}

\begin{remark}
Lemma \ref{rerootlemma} has a path interpretation.  The r.h.s. of
equation \eqref{reroot} is the generating function for paths on the
integer segment $[0,2]$, from vertex $0$ to $0$, with operator-valued weights:
$$
w(0\to 1)=w(1\to 2)=d, \quad w(2\to 1)=\bbB, \quad w(1\to 0)=\bbA
$$ The l.h.s. of equation \eqref{reroot} decomposes these paths into
the trivial one (length $0$, contribution $1$), and all the others,
which start with a step $0\to 1$ and end up with a step $1\to 0$, with
respective weights $d$ and $\bbA$ (in this order). In-between, we have
the generating function for ``rerooted" paths, from vertex $1$ to
vertex $1$, which consist of arbitrary sequences of either steps $1\to
0\to 1$ (with weight $\bbA d$) or steps $1\to 2\to 1$ (with
weight $d \bbB$). We call the rearrangement of this Lemma a ``rerooting''.
\end{remark}

\subsubsection{General case: mutations as rearrangements}\label{mutandis}

For each Motzkin path $\bM$, the solution of the $T$-system is can be
expressed in terms of the initial data at $\bx_\bM$ as
\begin{equation}
T_{1,j,k}=T_{1,j+k-m_1,m_1}\, \langle j-k+m_1\vert \bbF_\bM \vert
j+k-m_1\rangle
\end{equation}
for some operator continued fraction $\bbF_\bM$. Our main claim is
that this fraction is
obtained from $\bbF$ \eqref{opcont} via a succesion of applications of
Lemmas \ref{rerootlemma} and \ref{reargtlemma}.

For each Motzkin path we will define weights $\bbY_i(\bM)$, ($i\in
I_{2r+1}$), which are monomials in $\bx_\bM$ and $d$. The fraction $\bbF_\bM$ is a
function of these. As in \cite{DFK3}, we find that the effect of
mutations on
$\bbF_\bM\mapsto \bbF_{\bM'}$ is the following:
\begin{itemize}
\item If $\al=1$, use Lemma \ref{rerootlemma} to write
\begin{equation}\label{rerootone}
\bbF_\bM=1+d \bbF_{\bM}' \bbY_1(\bM).
\end{equation}
where $\langle t |\bbY_1(\bm)= T_{1,t,m_1+1}/T_{1,t+1,m_1}\langle t+1|$.
Then Lemma \ref{reargtlemma} enables us to rewrite $\bbF_{\bM}'$ as
$\bbF_{\bM'}$, a function of $\bx_{\bM'}$.
\item If $\al>1$, apply Lemma \ref{reargtlemma} to the part
of $F_\bM$ involving the weights $\bbY_\beta(\bM)$ with $\beta\geq 2\al-1$. 
\end{itemize}
In both cases, the weights $\bbY_\beta(\bM)$ are transformed into weights
$\bbY_\beta(\bM')$.

We will provide the precise construction and proof in Section
\ref{pathposit}, but for clarity, we refer the reader to Appendix \ref{atwo},
where the example of $A_2$ is worked out completely.

\subsection{Paths on graphs with non-commutative weights}\label{pathposit}

In this section, we define graphs with weights in $\cA[d,d^{-1}]$. The
generating functions $\bbF_\bM$ are path partition functions on these
graphs. The graphs are identical to those introduced in \cite{DFK3},
and the weights contain exactly the same information contained in the
two-dimensional representation of paths on these graphs introduced in
\cite{DFK3}. The construction presented here is therefore a rephrasing
of these paths in terms of operators.

\begin{remark} Although the two-dimensional representation of paths
  used here is identical to the one we used in \cite{DFK3}, we did
  not, in the earlier paper, have use for the full information
  contained in this path representation. In particular, the
  horizontal coordinate (``time'' in our language) had no
  interpretation in the context of $Q$-systems. Here, it corresponds to
  what is known as the spectral parameter in the $T$-system equations.
\end{remark}

Since one is used to reading lattice paths from left to
right, we have chosen to act on the space $V^*$
instead of $V$. In that way, the order in which they act on the space
is the same as the order the path is traversed.

\subsection{The target graphs $\Gamma_\bm$}
\begin{figure}
\centering
\includegraphics[width=8.cm]{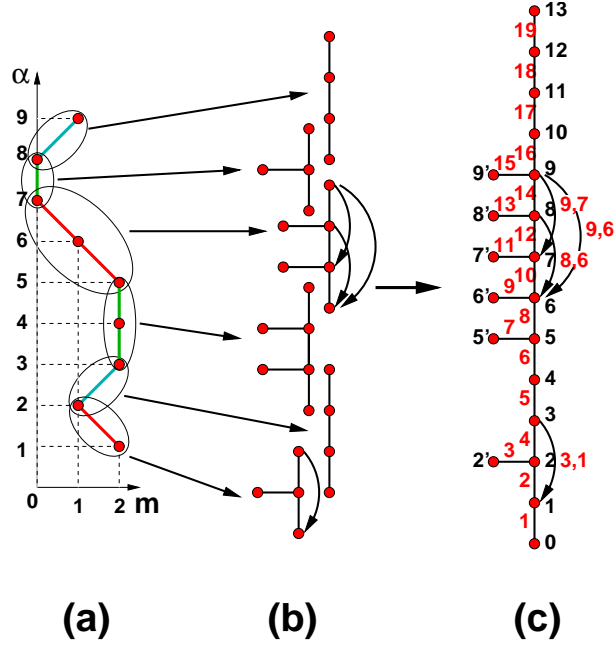}
\caption{\small The construction of the graph $\Gamma_\bM$ (c) for some 
sample Motzkin path $\bM$ (a) for $A_9$.
All the pieces glued (b) are represented vertically. The skeleton edges in (c) 
are labelled $1$ to $19$, the spine vertices
$0$ to $13$.}
\label{fig:glugraph}
\end{figure}

Let $\bM$ be a Motzkin path. We decompose it into pieces which do not
change direction: $\bM=\bM_1 \cup \bM_2 \cup \cdots$, where 
$\bM_i=(a,a+k, a+2k, \cdots, a+ (l_i-1) k)$ with $k=0,1$ or $-1$. The
{\em type} 
of subpath is called $k$. All the graphs used below must be drawn vertically
(see Fig.\ref{fig:glugraph}), which makes unambiguous the notion of top and 
bottom edges.

We construct a graph $\Gamma_{\bM_i}$ for each $i$ as follows:
\begin{itemize}
\item If $k=0$ then $\Gamma_{\bM_i}=\wG_{l_i}''$, the graph
  $\wG_{l_i}$ of Figure \ref{fig:dualgr} (represented vertically), with its bottom 
and top edges
  removed. 
\item If $k=1$, then $\Gamma_{\bM_i}$ is a simple (vertical) chain with $2l_i$
  vertices.
\item If $k=-1$, then $\Gamma_{\bM_i}$ is the graph $\wG_{l_i}''$
(represented vertically)
  decorated with additional oriented ``descending" edges $b\to a$ with
  $l_i+1\geq b>a+1>1$. 
\end{itemize}
We then glue the graphs: $\Gamma_{\bM_i}|\Gamma_{\bM_{i+1}}$ is the
graph obtained by identifying the top edge of $\Gamma_{\bM_i}$ with
the bottom edge of $\Gamma_{\bM_{i+1}}$.  Define
$\Gamma_\bM''=\Gamma_{\bM_1}|\Gamma_{\bM_{2}}|\cdots$, and
$\Gamma_{\bM}$ is $\Gamma_\bM''$ together with one additional
bottom and top edge and vertex. The graph $\Gamma_\bM$ is rooted at its 
bottom vertex.

Each graph thus constructed has a {\em spine}, namely a maximal vertical 
chain
of vertices consisting of the unprimed vertices of the various pieces glued.
We label the spine vertices
consecutively starting from $0$ at the bottom (see Fig.\ref{fig:glugraph}(c)). 
The vertices off the spine, which are attached only
to a vertex $i$ are labeled $i'$. We define
the {\em skeleton} of $\Gamma_\bM$ as the
graph with all edges $i\to j$ removed where $i>j+1$. The edges of the 
skeleton,
referred to as skeleton edges 
are labeled $1,2,...,2r+1$ from bottom to top (see Fig.\ref{fig:glugraph}(c)).

A path traverses each edge of a graph in one direction or another, and
in our formulation, we weight steps in each direction
differently. Therefore, we now consider the non-oriented edges in
$\Gamma_\bM$ as doubly-oriented edges, each orientation corresponding
to a different weight.

Assign a weight $\bbY_{a,b}(\bM)$ to the edge $a\to b$. The weight of any
edge away from the root, $\bbY_{i,i+1}=\bbY_{j,j'}=d$.  Skeleton edges
$\bbY_\al(\bM)$ pointing towards the root are independent weights,
which we define below. Weights on edges $i\to i-k$ with $k>1$ are
defined as the following product involving only skeleton weights or their inverses:
\begin{equation}\label{redun}
\bbY_{i,i-k}=\left(\prod_{a=i-1}^{i-k+1} \bbY_{a+1,a}\,
(\bbY_{a,a'})^{-1} (\bbY_{a',a})^{-1}\right)\,\bbY_{i-k+1,i-k}. 
\end{equation}
The ordered product is taken over edges from top to bottom, along a path from
vertex $i$ to $i-k$.  Note that 
$\bbY_{a,a'}=d$.  By inspection of \eqref{redun} we see that $\langle
t| \bbY_{i,i-k} \propto \langle t-k+2|$, hence $\bbY_{i,i-k}$ ``goes
back in time" by $k-2$ units.

\subsection{The positivity theorems for $\{T_{\al,j,k}\}$}

We now write $T_{1,j,k}$ as the partition function for paths on
$\Gamma_\bM$. The values of the skeleton weights are determined by
considering the effect of a mutation on the seed data -- They are
determined by a recursion relation, which can be solved explicitly.

\subsubsection{Transfer matrices and mutations}

As we illustrated in \cite{DFK3}, any Motzkin path has a unique
expression as a sequence of forward mutations, $m_\beta \mapsto
m'_\beta=m_\beta+\delta_{\beta,\alpha}$ where $\bM=(m_1,...,m_r)$ and 
$\bM'=(m_1',...,m_r')$. We restrict the mutations to those which increase $m_\al$ by $+1$
only in the following two cases:
\begin{itemize}
\item Case (i): $m_{\al-1}=m_\al=m_{\al+1}-1$,
\item Case (ii): $m_{\al-1}=m_\al=m_{\al+1}$,
\end{itemize}
(together with their boundary versions). This restricted set of mutations is sufficient 
to construct all Motzkin paths in the fundamental domain.

The initial step in the induction is the Motzkin path $\bM_0$.  The
path interpretation on $\Gamma_{\bM_0}={\wG}_r$ was given in Section
\ref{firpositsec}. The operator transfer matrix $\bbT_{\bM_0}=\bbT$
and the operator generating function $\bbF_{\bM_0}=\bbF$ are expressed
entirely in terms of the $d$ operator and the skeleton weights
$\bbY_\al(\bM_0)=\bbY_\al$ \eqref{skelweight}.

The inductive step is as follows. Given $\Gamma_\bM$ and its
operator weights, consider a forward mutation $\mu_\al$ or
$\mu_{\overline \al}$: $\bM\mapsto \bM'$. 
These have associated transfer matrices
$\bbT_\bM$ and $\bbT_{\bM'}$ corresponding to the graphs $\Gamma_\bM$
and $\Gamma_{\bM'}$.  We compare the associated generating
functions $\bbF_\bM=(I-\bbT_\bM)^{-1}_{0,0}$ and
$\bbF_{\bM'}=(I-\bbT_{\bM'})^{-1}_{0,0}$ using the row reduction
process. Both are operator continued fractions, which differ locally
due to the struction of the graphs. We find that the two operator
continued fractions are equal to each other if and only if the weights
of the graph $\Gamma_{\bM'}$ are related to those of $\Gamma_{\bM}$ as
follows.

\begin{thm}\label{weightrecu}
Let $\bbY'=\bbY(\bM')$ and $\bbY=\bbY(\bM)$, where
$\bM'=\mu_\al(\bM)$ or $\mu_{\overline{\al}}(\bM)$.
If $\al \neq 1$, then,
\begin{itemize}
\item{\bf Case (i): }
\begin{eqnarray}\label{weightncone}
\bbY_{2\al-1}'&=& \bbY_{2\al-1}+\bbY_{2\al}\nonumber \\
\bbY_{2\al}'&=& \bbY_{2\al+1} \bbY_{2\al} (\bbY_{2\al-1}')^{-1}\\
\bbY_{2\al+1}'&=&d^{-1}\bbY_{2\al+1}\bbY_{2\al-1} (\bbY_{2\al-1}')^{-1}d
\nonumber 
\end{eqnarray}
\item{\bf Case (ii): } in addition to the previous, we have
\begin{equation}\label{weightnctwo}
\bbY_{2\al+2}'=d^{-1}\bbY_{2\al+2}\bbY_{2\al-1} (\bbY_{2\al-1}')^{-1}d, \quad 
{\rm and}\ \ \bbY_\beta'=d^{-1}\bbY_\beta d,\ \forall \beta\geq 2\al+3
\end{equation}
\end{itemize}
If $\al=1$, we simply have to substitute $\bbY_1\to d^{-1}\bbY_1 d$ in the 
above formulas.
\end{thm}
\begin{proof}
The proof is by Gaussian elimination as in \cite{DFK3}.
The case $\al=1$ is special, as it requires a rerooting of
the generating function. The transformation of weights must be applied
on $\bbF'_{\bM}=(I-\bbT_\bM)^{-1}_{1,1}$ as in \eqref{rerootone},
which induces the substitution $\bbY_1\to d^{-1}\bbY_1 d$.
\end{proof}

Let
\begin{equation}\label{deflamu}
\lambda_{\al,t,m}={T_{\al,t,m+1}\over T_{\al,t+1,m}},\qquad 
\mu_{\al,t,m}={T_{\al,t,m}\over T_{\al-1,t+1,m}}
\end{equation}
\begin{cor}\label{solweight}
The skeleton weights obeying the recursions of Theorem
\ref{weightrecu}, subject to the initial condition \eqref{skelweight}
are the operators $\bbY_\beta(\bM)$, acting as $\langle t|
\bbY_\beta(\bM)=y_\beta(\bM;t)\langle t+1|$, with:
\begin{eqnarray}\label{weightsol}
\ \ y_{2\al -1}(\bM;t)&=& 
{\lambda_{\al,t+m_\al-m_1-1,m_\al}\over \lambda_{\al-1,t+m_{\al-1}-
m_1,m_{\al-1}}} \\
y_{2\al}(\bM;t)&=& 
{\mu_{\al+1,t+m_\al-m_1,m_\al +1}\over \mu_{\al,t+m_\al-m_1,m_\al}}\times 
\left\{ \begin{matrix}
{\lambda_{\al+1,t+m_{\al+1}-m_1,m_{\al+1}}\over \lambda_{\al+1,t+m_{\al}-
m_1,m_{\al}}}
 & {\rm if} \ m_{\al}=m_{\al+1}+1 \\
1  & {\rm otherwise} \end{matrix} \right\} \\
&&\times \left\{ \begin{matrix}
{\lambda_{\al-1,t+m_{\al}-m_1,m_{\al}}\over \lambda_{\al-1,t+m_{\al-1}-
m_1,m_{\al-1}}}
 & {\rm if} \ m_{\al}=m_{\al-1}-1 \\
1  & {\rm otherwise}
\end{matrix} \right\}
\end{eqnarray}
\end{cor}
\begin{proof}
By direct check of the recursion relations
(\ref{weightncone}-\ref{weightnctwo}).
\end{proof}

Thus, we have two expressions for the generating function of
$T_{1,j,k}$, one in terms of the seed data $\bx_\bM$ and the other in
terms of the seed data $\bx_{\bM'}$. We call the transition between the two
expressions a mutation: It acts on the graph $\Gamma_{\bM}$ and on its
weights. Alternatively, it acts on the operator continued fraction
expresson for $\bbF_\bM$ as a rearrangement.

\subsubsection{Positivity of $T_{1,j,k}$}

We note that the weights \eqref{weightsol} $y_\al(\bM;t)$ are
positive Laurent monomials of the initial data at $\bx_\bM$.  We
therefore have a positivity result:

\begin{thm}\label{pathone}
$T_{1,j,k+m_1}/T_{1,j+k,m_1}$ is the partition function for paths on
the rooted graph $\Gamma_\bM$ with the weights of Theorem
\ref{solweight}, starting from the root at time $j-k$ and ending at the root at 
time $j+k$. 
As such it is a positive Laurent polynomial of the
mutated data at $\bx_\bM$.
\end{thm}

\subsubsection{General solution and strongly non-intersecting paths}

We now turn to the expression of $T_{\al,j,k}$ in terms of the mutated
initial data $\bx_\bM$.  We will interpret the
determinant formula Theorem \ref{elim} for $T_{\al,j,k}$ \`a la
Gessel-Viennot, in terms of the strongly non-intersecting paths on
the graph $\Gamma_\bM$ introduced in \cite{DFK3}.

\begin{figure}
\centering
\includegraphics[width=8.cm]{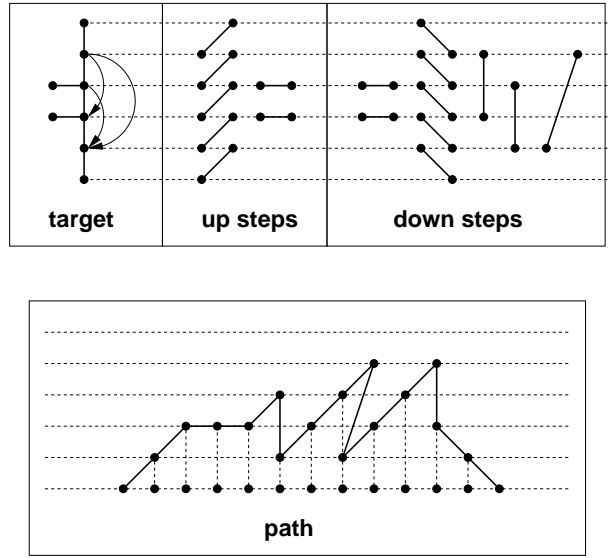}
\caption{\small Two-dimensional lattice path representation
on the graph $\Gamma_\bM$, $\bM=(2,1,0)$ of the $A_3$ case.
We have indicated the ``up" steps (i.e. away from the root) and the ``down" 
steps
(towards the root).}
\label{fig:pathslope}
\end{figure}

Let us briefly recall the two-dimensional $\Gamma_\bM$-lattice paths 
used to represent paths on
$\Gamma_\bM$, given in \cite{DFK3}. There are fundamentally three
kinds of oriented edges in $\Gamma_\bM$: the horizontal and vertical 
``skeleton" edges,
and down-pointing long edges, with weights which depend on the skeleton weights. 
The steps taken along these edges on $\Gamma_\bM$ are represented in $
\Z^2$
as follows (see Fig.\ref{fig:pathslope} for an illustration):
\begin{itemize}
\item A skeleton step $i\to i+\epsilon$, $\epsilon=\pm 1$, 
at time $t$ becomes the segment from $(t,i)$ to $(t+1,i+\epsilon)$
\item A skeleton step $i\to i'$ or $i'\to i$ at time $t$ becomes the
segment from $(t,i)$ to $(t+1,i)$
\item A long step $j\to i$, $j>i+1$ at time $t$ becomes the segment from $(t,j)$ to $(t+j-i-2,i)$
\end{itemize}
Note that the increment of $x$-coordinate for each step coincides with the 
time shift we have associated with each step. Indeed, all steps advance by 
one unit
of time, except the long ones, which go back in time by $j-i-2\leq 0$.
The only difference with \cite{DFK3} is that we now attach 
{\em time-dependent} 
weights to the steps namely a weight $y_{a,b}(t)$ for a step $a\to b$ starting at 
time $t$
(In the operator language, we have operator weights $\bbY_{a,b}$ that act as
$\langle t| \bbY_{a,b}= y_{a,b}(t) \langle t+h |$, where $h$ is the time-shift of 
the 
corresponding step, $h=+1$ for all steps except the long ones, for which
$h=a-b-2$.). We conclude that this representation is perfectly adapted to our 
weighted paths, as the $x$-coordinate is nothing but the time-coordinate.

Let us consider $T_{\al,j,k}$ as a function of the initial data at $\bx_\bM$.
Writing 
\begin{equation}\label{deterpath}
{T_{\al,j,k+m_1}\over \prod_{b=1}^\al T_{1,j+k+2b-\al,m_1} } 
=\det_{1\leq a,b \leq \al} 
{T_{1,j-a+b,k+a+b-\al-1+m_1}\over T_{1,j+k+2b-\al-1,m_1} }
\end{equation}
Using Theorem \ref{pathone},
we may interpret $T_{1,j-a+b,k+a+b-\al-1+m_1}/T_{1,j+k+2b-\al-1,m_1}$
as the partition function for $\Gamma_\bM$-lattice paths from $s_a=(j-k+\al
+1-2a,0)$ to 
$e_b=(j+k+2b-\al-1,0)$. The determinant is simply a signed sum of products of 
such
path partition functions, 
corresponding in turn to the partition function for families of 
paths starting at $\{s_a\}_{a=1}^\al$
and ending at $\{e_b\}_{b=1}^\al$, with the usual weights 
times the signature  of 
the permutation of endpoints induced by the configuration.

\begin{figure}
\centering
\includegraphics[width=12.cm]{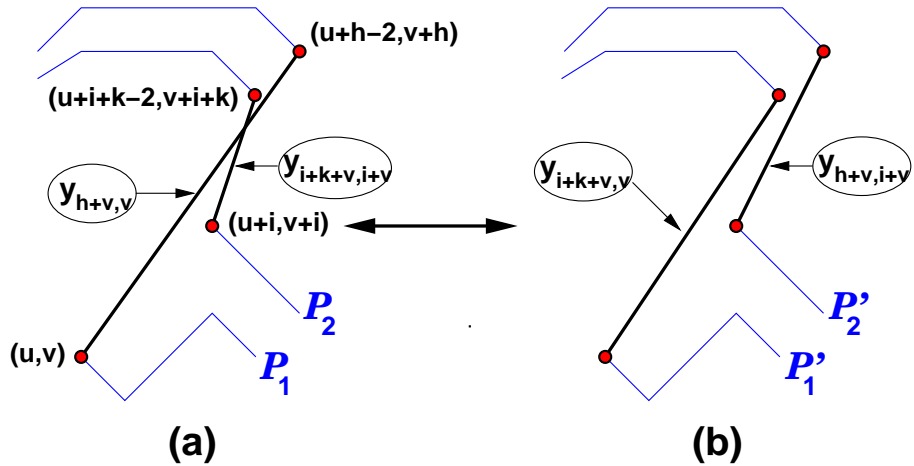}
\caption{\small A typical edge intersection of $\Gamma_\bM$-paths
(a) and the result of the flipping operation on it (b). We have indicated the
weights of the steps. 
%We have the identity of time-dependent path weights, 
%since
%$y_{h+v,v}(u+h-2)y_{i+k+v,i+v}(u+i+k-2)
%= y_{h+v,i+v}(u+h-2)y_{i+k+v,v}(u+i+k-2)$.  
The paths in (b)
are said to be ``too close'' to each other.}\label{fig:flipping}
\end{figure}

In the standard Gessel-Viennot case, these signs produce the necessary 
cancellations
to only leave us with the contribution of non-intersecting paths, namely families 
in which no
two paths share a vertex. This is best proved by introducing a sign-reversing 
involution that pairs up
and cancels all the unwanted terms in the expansion of the determinant. 

In the case of 
$\Gamma_\bM$-paths, the situation is more subtle, as paths may intersect 
without sharing a vertex.
In \cite{DFK3}, we have produced an involution, which allows to interpret an 
analogue of the determinant
\eqref{deterpath} as the partition function of strongly non-intersecting $
\Gamma_\bM$-lattice paths.
This involution consists in flipping paths as follows. We consider the first 
intersection between two
paths within a family. If the intersection is at a common vertex, we interchange 
the portions of paths
before the intersection. If it is not at a common vertex, we flip the two paths 
as indicated in Fig.\ref{fig:flipping}, by switching their beginnings until the crossing.

We make then the following crucial observation:
\begin{lemma}\label{crucialem}
In the generic flipping situation of Fig.\ref{fig:flipping},
the flipped pair of paths has the same (time-dependent) weight
as the original one, up to the sign of the permutation of starting points, 
due to the following relation:
$$
y_{h+v,v}(u+h-2)y_{i+k+v,i+v}(u+i+k-2)= 
y_{h+v,i+v}(u+h-2)y_{i+k+v,v}(u+i+k-2)
$$
\end{lemma}
\begin{proof} By direct application of the formula \eqref{redun} for the long edge weights.
\end{proof}

The only invariant families under this involution are those where the paths do 
not lie 
``too close" to each-other,
as otherwise they get cancelled by applying a flip.

Therefore all the conclusions of \cite{DFK3} still hold in the
present case, and we have:

\begin{thm}
$T_{\al,j,k+m_1}/\prod_{b=1}^\al T_{\al,j+k+2b-\al-1,m_1}$ is the partition 
function for
configurations of $\al$ strongly non-intersecting $\Gamma_\bM$-lattice paths, 
with the weights of Theorem \ref{solweight}. As
such, $T_{\al,j,k+m_1}$ is a positive Laurent polynomial of the mutated data at
$\bx_\bM$.
\end{thm}

\section{Conclusion}

The $T$-system equations are a special case of a
{\em non-commutative} $Q$-system. That is, one can write a $Q$-system
for non-commutative variables such that its matrix elements coincide with
the $T$-system equations. One can think of the non-commutative $Q$-system
as a ``non-commutative" cluster algebra.
Special cases of non-commutative cluster algebras have been
considered in several contexts, for example the quantum cluster algebras of
Berenstein and Zelevinsky \cite{BZ}, or the more general recursion relation
introduced by Kontsevich \cite{KPC} (in rank 2), with similar Laurent properties,
and which can be solved in some special ``affine" cases
using our methods \cite{DFKnew}. The main idea is that the path formulation
seems to be particularly well adapted to the explicit solution of such problems,
and makes Laurentness and  positivity manifest. This will be discussed in a future
publication.

We should mention that there has been a great deal of interest in 
$T$-systems with more restrictive boundary conditions \cite{N,HL}. 
We hope that the
construction introduced in this paper will provide a simple way of
treating such boundary conditions and their consequences.

\begin{appendix}
\section{Discrete Wronskians}

\subsection{Pl\"ucker relations}

Let $P$ be an
$N\times (N+k)$-matrix. Let $|P^{b_1,...,b_k}|$ be the determinant of
the matrix obtained by deleting the $k$ columns $b_1,...,b_k$ of $P$, times 
the signature of the permutation that
reorders these column indices in increasing order. Then we have:
\begin{equation}\label{plucker}
\vert P^{a_1,...,a_k}\vert \,\vert P^{b_1,...,b_k}\vert=\sum_{p=1}^k
\vert P^{b_p,a_2...,a_k}\vert  
\, \vert P^{b_1,...,b_{p-1},a_1,b_{p+1},...,b_k}\vert .
\end{equation}
for any choice of $2k$ columns $a_1,...,a_k$ and $b_1,...,b_k$ of $P$.
In particular, when $k=2$, we have
\begin{equation}\label{specplu}
  \vert P^{a_1,a_2}\vert \, \vert P^{b_1,b_2}\vert =\vert P^{b_1,a_2}\vert \, 
  \vert P^{a_1,b_2}\vert+\vert P^{b_2,a_2}\vert \, \vert P^{b_1,a_1}\vert.
\end{equation}
for any $N\times(N+2)$ matrix $P$. 

Equation \eqref{specplu}  implies the Desnanot-Jacobi relation. Let $M$ be an 
$N\times N$ matrix, and let $|M|$, $M_{i}^{j}$, $|M_{i_1,i_2}^{j_1,j_2}|$ denote 
the determinants of $M$,
the minor obtained by erasing row $i$ and column $j$ of $M$, and the double 
minor obtained
by erasing rows $i_1,i_2$ and columns $j_1,j_2$ of $M$, respectively. Let 
$1\leq i_1<i_2\leq N$ and $1\leq j_1<j_2\leq N,$ then
\begin{equation}\label{desna}
|M| \, |M_{i_1,i_2}^{j_1,j_2}|= |M_{i_1}^{j_1}|\, |M_{i_2}^{j_2}| - |M_{i_1}^{j_2}|\, 
|M_{i_2}^{j_1}|
\end{equation}
It is easily obtained as a particular case of eq.\eqref{specplu}, for 
$a_2=j_1$, $b_2=j_2$, $P_{i,a_1}=\delta_{i,i_1}$,
$P_{i,b_1}=\delta_{i,i_2}$, and $M$ is the matrix $P$ with columns $a_1$ and 
$b_1$ erased.
Indeed, one checks directly that: $|P^{a_1,b_1}|=|M|$, $|P^{a_2,b_2}|=|
M_{1,N}^{1,N}|$,
$|P^{a_1,a_2}|=|M_{N}^{1}|$, $|P^{b_2,b_1}|=-|M_{1}^{N}|$, $|P^{a_1,b_2}|=|
M_{N}^{N}|$,
and $|P^{b_1,a_2}|=-|M_{1}^{1}|$.

\subsection{$T$-system as discrete Wronskians}\label{wronskianApp}
Here present the proof of Theorem \ref{elim} which uses the relations in the 
previous subsection.
\begin{proof}
Consider Equation \eqref{desna}. Let $N=\al+1$, $i_1=j_1=1$, $i_2=j_2=N$ 
and choose the matrix 
$M$ with entries $M_{a,b}=T_{1,j+a-b,k+a+b-\al-2}$
for $a,b=1,2,...,\al+1$. We denote by $W_{\al+1,j,k}=|M|$ the corresponding 
``discrete Wronskian" determinant. Substituting this definition into eq.
\eqref{desna}, we 
have
\begin{equation}\label{relaW}
W_{\al+1,j,k} W_{\al-1,j,k} =W_{\al,j,k-1}W_{\al,j,k+1} -W_{\al,j-1,k}W_{\al,j+1,k}
\end{equation}
valid for $\al\in I_r$, provided we set $W_{0,j,k}=1$.
Note that $W_{1,j,k}=T_{1,j,k}$ by definition. Comparing eq.\eqref{relaW}
with the $T$-system \eqref{Tsys}, we deduce that the  $T$'s and $W$'s obey 
the
same recursion relations and share the same initial conditions at $\al=0$ and 
$1$.
As the system is a three-term recursion in $\al$ this determines the solution 
uniquely
and therefore we have $W_{\al,j,k}=T_{\al,j,k}$ for all $\al\in I_r$,
$j,k\in \Z$.
\end{proof}

\subsection{Linear recursion relations}\label{linearrecursion}
Here, we present a proof of Theorem \ref{recut}, that the variables $T_{\al,j;k}$ 
satisfy linear recursion relations, with constant coefficients which are the 
conserved quantities.
\begin{proof}
We perform the discrete analog of differentiating the Wronskian, and compute
$\varphi_{j-1,k+1}-\varphi_{j,k}=0$. Denoting by 
$\varphi_{j,k}=|\bg_1, \bg_2, \cdots ,\bg_{r+1}|$ and $\varphi_{j-1,k+1}=|\vf_1, 
\vf_2, \cdots ,\vf_{r+1}|$
as the determinants of column vectors $\bg_i,\vf_i$, we note that $\bg_{i+1}=
\vf_i$
for $i=1,2,...,r$. We may therefore rewrite
\begin{equation}
\varphi_{j-1,k+1}-\varphi_{j,k}=0= | \vf_1, \vf_2, \cdots \vf_r,\vf_{r+1}-(-1)^{r} 
\bg_{1}|
\end{equation}
with $(\vf_b)_a = T_{1,j-1+a-b,k+a-r-1+b} $ and $(\bg_{1})_a=T_{1,j-1+a,k+a-
r-1}$.
Therefore, there must exist a non-trivial linear combination of the columns of 
the matrix which vanishes. We write it as
\begin{equation}
\sum_{b=1}^r (-1)^b c_{r+1-b}(j,k) \vf_b + c_0(j,k)\,(\vf_{r+1}-(-1)^r \bg_1) =0.
\end{equation}

Recall that the entries of the vectors $\vf_b$ depend on $j,k$ in a very
particular way, namely $(\vf_b(j,k))_{a+1} =(\vf_b(j+1,k+1))_{a}$, and similarly 
for $\bg_1$.
In order for the above linear combination to be non-trivial, we must therefore 
have
$c_b(j,k)=c_b(j-1,k-1)=\cdots=c_b(j-k,0)$ 
for all $j,k\in \Z$, hence the coefficients
$c_b$ only depend on the difference $j-k$. Finally, we may normalize the 
coefficients
in such a way that $c_0=1$ identically, and the first part of the Theorem 
follows.

The second part is treated analogously, by considering the difference of 
Wronskians
$\varphi_{j+1,k+1}-\varphi_{j,k}=0$
and reasoning on the rows of the corresponding matrices.
\end{proof}

\subsection{Conserved quantities as Wronskian determinants with defect}
\label{conservedwronskian}
Here, we give the proof of Lemma \ref{conwron} expressing the conserved 
quantities of the $T$-system as 
Wronskian determinants with defects.

\begin{proof}
Let  $\gamma_m(j,n)$ denote the right hand side of of eq.\eqref{consW}.

It is clear that that $\gamma_0(j,n)=T_{r+1,j+n,n+r}=1$ and $\gamma_{r+1}
(j,n)=T_{r+1,j+n-1,n+r+1}=1$
as consequences of Theorem \ref{elim} and of  the $A_r$ boundary condition. 

Let 
$p\in\Z$ and define the
 $(r+2)\times (r+2)$ matrix $D$  to be the matrix with entries
$D_{1,b}=T_{1,j+p+1-b,p+b-1}$, and $D_{a,b}=T_{1,j+n+a-b,n+a+b-2}$ for 
$a=2,3,...,r+2$ 
and $b=1,2,...,r+2$. The identity \eqref{linrecone} may be recast into a 
vanishing non-trivial linear
combination of the columns of $D$, with coefficients $c_{r+2-b}(j)(-1)^{b-1}$, 
$b=1,2,...,r+2$, hence the determinant
of $D$ vanishes. 

Expanding the determinant along the first row, we find that
\begin{equation}\label{reclin}
0=\det(D)=\sum_{b=1}^{r+2} (-1)^{b+1}D_{1,b} |D_1^b|=\sum_{b=1}^{r+2}
(-1)^{b-1}
\gamma_{r+2-b}(j,n) T_{1,j+p+1-b,p+b-1} ,
\end{equation}
as the determinants $\gamma_m(j,n)$
are  the minors $|P_{1}^{r+2-m}|$. 

Since  the Wronskian
determinant $T_{r+1,j,k}=1$ is non-zero, there exist no other non-trivial linear 
recursion relation
than Equation \eqref{linrecone} with strictly fewer terms, hence the coefficients 
in Equation \eqref{reclin}
must be proportional to those in Equation \eqref{linrecone}. As $c_0(j)=
\gamma_0(j,n)=1$,
we deduce that $\gamma_m(j,n)=c_m(j)$ for all $m=0,1,2...,r+1$, and the 
Lemma follows.
\end{proof}

\section{Example of $A_2$: rearrangements, graphs and paths}\label{atwo}

Here, we illustrate the program of Section \ref{mutandis} in the case $r=2$. We 
first
present the rearrangements of the operator continued fraction $F$,
which make positivity of $R_{1,n}$ manifest in all three cases. Next,
we interpret these in terms of partition functions for
operator-weighted paths on graphs, to illustrate Section \ref{pathposit}.

\subsubsection{Rearrangements}
The fundamental domain for the action of mutations on the fundamental seed 
$\wx_0$
is coded by the following three Motzkin paths with 2 vertices: $\bm_0=(0,0)$, 
$\bm_1=\mu_1(\bm_0)=(1,0)$ and $\bm_2=\mu_2(\bm_0)=(0,1)$. We give 
below the three
operator continued fractions corresponding to these points.

{\bf Seed $\bx_0$:} The continued fraction $F_0(\by)$
reads for the fundamental seed corresponding to
the Motzkim path $\bm_0$ is:
\begin{equation}
F_0(\by)=\left(1-d \Big(1-d \big(1-d \bbY_3-d(1-d \bbY_5)^{-1}
\bbY_4\big)^{-1}\bbY_2\Big)^{-1} \bbY_1\right)^{-1}
\end{equation}
with operators $\bbY_i$, $i\in I_5$, acting as  $\langle t|\bbY_i=y_i(t)\, \langle t
+1|$, and:
\begin{eqnarray*}
&&y_1(t)={T_{1,t,1}\over T_{1,t+1,0} },\quad 
y_2(t)={T_{2,t,1}\over T_{1,t,0}T_{1,t+1,1} },\quad
y_3(t)={T_{1,t+1,0}T_{2,t-1,1}\over T_{1,t,1}T_{2,t,0} } \\
&&y_4(t)={T_{1,t+1,0}\over T_{2,t,0}T_{2,t+1,1} },\quad 
y_5(t)={T_{2,t+1,0}\over T_{2,t,1} }
\end{eqnarray*}

{\bf Seed $\bx_2=\mu_2(\bx_0)$:} Following Section \ref{mutandis},
we apply the Lemma \ref{reargtlemma} to $F_0$, with 
$\bbA=\bbY_3$, $\bbB=\bbY_4$, $\bbC=\bbY_5$ and $\bbU=0$:
This yields $F_0(\by)=F_2(\bw)$, where
 
\begin{equation}
F_2(\bw)= \left(1-d\Bigg(1-d \Big(1-d \big(1-d(1-d 
\bbW_5)^{-1}\bbW_4 \big)^{-1}\bbW_3 \Big)^{-1}\bbW_2 \Bigg)^{-1}\bbW_1 
\right)^{-1}
\end{equation}
with operators $\bbW_i$, $i\in I_5$, acting as $\langle t|\bbW_i=w_i(t) \langle t
+1|$, with:
\begin{eqnarray*}
&&w_1(t)={T_{1,t,1}\over T_{1,t+1,0} },\quad 
w_2(t)={T_{2,t,1}\over T_{1,t,0}T_{1,t+1,1} },\quad
w_3(t)={T_{1,t+1,0}T_{2,t,2}\over T_{1,t,1}T_{2,t+1,1} } \\
&&w_4(t)={T_{1,t+1,1}\over T_{2,t,1}T_{2,t+1,2} },\quad 
w_5(t)={T_{2,t+1,1}\over T_{2,t,2} }
\end{eqnarray*}

To obtain this, we have written $\bbW_1=\bbY_1$, $\bbW_2=\bbY_2$, and $
\bbW_3=\bbY_3+\bbY_4$,
while $\bbW_4=\bbY_5 \bbY_4 \bbW_3^{-1}$ and $\bbW_5=d^{-1}\bbY_5 
\bbY_3 \bbW_3^{-1} d$,
and used the $T$-system to simplify the expressions.

{\bf Seed $\bx_1=\mu_1(\bx_0)$:} Following Section \ref{mutandis}, we first
apply the rerooting Lemma \ref{rerootlemma}, with $\bbA=\bbY_1$ and $\bbB=
\bbV \bbY_2$,
where 
$$
\bbV= \left( 1-d \bbY_3-d (1-d \bbY_5)^{-1} \bbY_4\right)^{-1}
$$
This allows to rewrite $F_0(\by)=1+d F_0'(\by) \bbY_1$, with
$F_0'(\by)=(1-\bbY_1 d-d \bbU  \bbY_2)^{-1}$. We may now apply the 
rearrangement Lemma
\ref{reargtlemma}, with $\bbA=d^{-1}\bbY_1 d$, $\bbB= \bbY_2$, $\bbC=
\bbY_3+(1-d \bbY_5)^{-1} \bbY_4$,
and $\bbU=0$: this yields $F_0(\by)=1+d F_1(\bz) \bbY_1$, where:
\begin{equation}
F_1(\bz)= \left(1-d\Bigg(1-d \Big(1-d\bbZ_3 -d (1-d 
\bbZ_5)^{-1}\bbZ_4 \Big)^{-1}\Big(\bbZ_2 +(1-d \bbZ_5 )^{-1} \bbZ_6 \Big) 
\Bigg)^{-1}\bbZ_1\right)^{-1}
\end{equation}
where $\bbZ_i$, $i\in I_5$ act as $\langle t|\bbZ_i=z_i(t) \langle t+1|$, with:
\begin{eqnarray*}
&&z_1(t)={T_{1,t,2}\over T_{1,t+1,1} },\quad 
z_2(t)={T_{2,t-1,1}T_{2,t+1,1}\over T_{1,t,1}T_{1,t+1,2}T_{2,t,0} },\quad
z_3(t)={T_{1,t+1,1}T_{2,t-2,1}\over T_{1,t,2}T_{2,t-1,0} } \\
&&z_4(t)={T_{1,t-1,1}T_{1,t+1,1}\over T_{2,t-1,0}T_{2,t,1}T_{1,t,2} },\quad 
z_5(t)={T_{2,t,0}\over T_{2,t-1,1} }
\end{eqnarray*}
and $\bbZ_6$ is a long step weight, expressed in terms of the skeleton weights as:
$\bbZ_6=\bbZ_4 (d \bbZ_3)^{-1} \bbZ_2$ (a particular case of Eq.\eqref{redun}). 
Note that it is diagonal, namely:
$$
\langle t | z_6 =z_6(t) \langle t |,\ \  {\rm where}\qquad z_6(t)={z_4(t)\over 
z_3(t)} z_2(t-1)=
{1\over T_{1,t,2}T_{2,t-1,0}}
$$
The above weights follow from the identifications: $\bbZ_1=d^{-1}\bbY_1 d+
\bbY_2$, 
$\bbZ_2=\bbY_3  \bbY_2 \bbZ_1^{-1}$, $\bbZ_6=d^{-1}\bbY_4 \bbY_2 
\bbZ_1^{-1}$,
$\bbZ_3=d^{-1}\bbY_3 d^{-1}\bbY_1d \bbZ_1^{-1}$, $\bbZ_4=d^{-1}\bbY_4 
d^{-1}\bbY_1d \bbZ_1^{-1}$ 
and $\bbZ_5=d^{-1}\bbY_5d$,
and the use of the $T$-system to simplify the expressions.

\subsubsection{Paths on graphs with operator weights}
\begin{figure}
\centering
\includegraphics[width=6.cm]{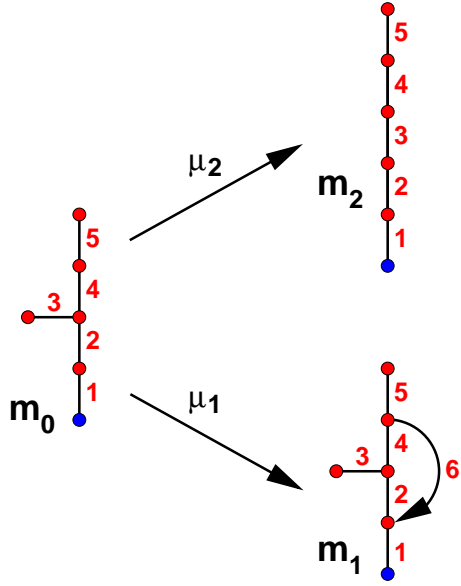}
\caption{The fundamental domain for $A_2$, coded by the Motzkin paths 
$\bm_0,\bm_1,\bm_2$, and the corresponding target graphs for the path 
interpretation,
with their edge labels.
Mutations are indicated by arrows.}\label{fig:pathslthree}
\end{figure}

Recall first that the
continued fractions $F_0(\by),F_2(\bw)$ are such that 
$T_{1,j,k}/T_{1,j+k,0}=\langle j-k | F_i | j+k\rangle$, $i=0,2$, while, due to the 
re-rooting, we
have $T_{1,j,k}/T_{1,j+k-1,1}=\langle j-k+1 | F_1 | j+k-1\rangle$. 

The three above operator continued fractions may be interpreted in terms of 
path counting
as follows. We have represented in Figure \ref{fig:pathslthree} the three rooted 
target graphs
$\Gamma_\bm$, attached to the three Motzkin paths
$\bm=\bm_0,\bm_1,\bm_2$, together with their edge labelings. 
We have the following 

\begin{thm}
For $i=0,1,2$,
the quantities $\langle t| F_i | t'\rangle$ are the partition functions for paths 
on the graphs $\Gamma_{\bm_i}$, from and to the root, starting at time $t$
and ending at time $t'$, and with operator weights defined as the product over
the operator weights for each successive step of the path, in the same order. 
The weights
are $d$ per step away form the root, and respectively $\bbY_i$, $\bbZ_i$ and 
$\bbW_i$
per step towards the root, along the edge labeled $i$.
\end{thm}
\begin{proof}
The proof is a straightforward adaptation of the argument of Section 
\ref{pathcfrac}: it uses
operator transfer matrices $\bbT_i$, and amounts to performing the Gaussian 
elimination
of $I-\bbT_i$, in order to compute $\left((I-\bbT_i)^{-1}\right)_{0,0}$, where $0$ 
indexes the root vertex on $\Gamma_{\bm_i}$.  We give explicit expressions 
below.
\end{proof}

We now list the transfer matrices for the three cases above. In all cases, we 
have
$F_i =\left((I-\bbT_i)^{-1}\right)_{0,0}$.

$$\small 
\bbT_0=\begin{pmatrix} 
0 & d & 0  & 0 & 0 & 0\\
\bbY_1 & 0 & d & 0 & 0 & 0\\
0 & \bbY_2 & 0 & d & 0 & 0\\
0 & 0 & \bbY_3 & 0 & d & d \\
0 & 0 & \bbY_4 & 0 & 0 & 0 \\
0 & 0 & 0 & 0 & \bbY_5 & 0
\end{pmatrix}
\bbT_1=\begin{pmatrix} 
0 & d & 0  & 0 & 0 & 0\\
\bbZ_1 & 0 & d & 0 & 0 & 0\\
0 & \bbZ_2 & 0 & d & 0 & 0\\
0 & 0 & \bbZ_3 & 0 & d & d \\
0 & \bbZ_6 & \bbZ_4 & 0 & 0 & 0 \\
0 & 0 & 0 & 0 & \bbZ_5 & 0
\end{pmatrix}
\bbT_2=\begin{pmatrix} 
0 & d & 0  & 0 & 0 & 0\\
\bbW_1 & 0 & d & 0 & 0 & 0\\
0 & \bbW_2 & 0 & d & 0 & 0\\
0 & 0 & \bbW_3 & 0 & d & 0 \\
0 & 0 & 0 & \bbW_4 & 0 & d \\
0 & 0 & 0 & 0 & \bbW_5 & 0
\end{pmatrix}
$$

\begin{remark}
Note that, as opposed to the two other cases,
the transfer matrix $\bbT_1$ is not made of diagonal operators times $d$, as
$\bbZ_6$ is diagonal, hence goes back one step in time compared to the other 
operators $\bbZ_i$, $i=1,2,...,5$. This necessity for the longer descending 
steps
to go back in time was already observed in \cite{DFK3} in the two-dimensional
representation of the $\Gamma_\bm$-paths.
\end{remark}

\end{appendix}

\def\cprime{$'$} \def\cprime{$'$} \def\cprime{$'$} \def\cprime{$'$}
  \def\cprime{$'$} \def\cprime{$'$} \def\cprime{$'$} \def\cprime{$'$}
  \def\cprime{$'$} \def\cprime{$'$}
\providecommand{\bysame}{\leavevmode\hbox to3em{\hrulefill}\thinspace}
\providecommand{\MR}{\relax\ifhmode\unskip\space\fi MR }
\providecommand{\MRhref}[2]{%
  \href{http://www.ams.org/mathscinet-getitem?mr=#1}{#2}
}
\providecommand{\href}[2]{#2}

\end{document}